\documentclass[a4paper,10pt]{article}

\usepackage[english]{babel}
\usepackage{amsmath,amsfonts,amssymb,amsthm}
\usepackage{mathrsfs}
\usepackage{graphicx}
\usepackage{bbm}
\usepackage{indentfirst}

\newtheorem{theorem}{Theorem}[section]

\newtheorem{lemma}{Lemma}[section]

\newtheorem{corollary}{Corollary}[section]

\theoremstyle{definition}
\newtheorem{remark}{Remark}[section]

\numberwithin{equation}{section}

\newcommand\blfootnote[1]{\begingroup\renewcommand\thefootnote{}\footnote{#1}\addtocounter{footnote}{-1}\endgroup}

\begin{document}

\title{
{\bf\Large
Multiple positive solutions for a superlinear problem: a topological approach }\footnote{Work performed under the auspices of the
Grup\-po Na\-zio\-na\-le per l'Anali\-si Ma\-te\-ma\-ti\-ca, la Pro\-ba\-bi\-li\-t\`{a} e le lo\-ro
Appli\-ca\-zio\-ni (GNAMPA) of the Isti\-tu\-to Na\-zio\-na\-le di Al\-ta Ma\-te\-ma\-ti\-ca (INdAM).}
\vspace{1mm}}

\author{
{\bf\large Guglielmo Feltrin}
\vspace{1mm}\\
{\it\small SISSA - International School for Advanced Studies}\\
{\it\small via Bonomea 265}, {\it\small 34136 Trieste, Italy}\\
{\it\small e-mail: guglielmo.feltrin@sissa.it}\vspace{1mm}\\
\vspace{1mm}\\
{\bf\large Fabio Zanolin}
\vspace{1mm}\\
{\it\small Department of Mathematics and Computer Science, University of Udine}\\
{\it\small via delle Scienze 206},
{\it\small 33100 Udine, Italy}\\
{\it\small e-mail: fabio.zanolin@uniud.it}}

\date{}

\maketitle

\vspace{-2mm}

\begin{abstract}
\noindent
We study the  multiplicity of positive solutions for a two-point boundary value problem
associated to the nonlinear second order equation $u''+f(x,u)=0$.
We allow $x \mapsto f(x,s)$ to change its sign in order to cover the case of
scalar equations with indefinite weight.
Roughly speaking, our main assumptions require that $f(x,s)/s$ is below $\lambda_{1}$ as $s\to 0^{+}$
and above $\lambda_{1}$ as $s\to +\infty$. In particular, we can deal with the situation in which
$f(x,s)$ has a superlinear growth at zero and at infinity.
We propose a new approach based on the topological degree which provides the multiplicity of solutions.
Applications are given for $u'' + a(x) g(u) = 0$, where we prove the existence of $2^{n}-1$
positive solutions when $a(x)$ has $n$ positive humps and $a^{-}(x)$ is sufficiently large.
\end{abstract}

\section{Introduction}\label{section1}

In this paper we study the problem of existence and multiplicity
of positive solutions for the nonlinear two-point boundary value problem
\begin{equation}\label{two-pointBVP0}
\begin{cases}
\, u''+f(x,u)=0 \\
\, u(0)=u(L)=0,
\end{cases}
\end{equation}
where $f\colon \mathopen{[}0,L\mathclose{]}\times{\mathbb{R}}^{+} \to {\mathbb{R}}$ is a Carath\'{e}odory function such that $f(x,0)\equiv 0$.
As a main application of our results for \eqref{two-pointBVP0} we consider the case in which
\begin{equation}\label{eq-1.1}
f(x,s) = a(x) g(s),
\end{equation}
with $g(s)/s \to 0$ as $s\to 0^{+}$ and $g(s)/s \to +\infty$ as $s\to +\infty$,
thus covering the classical superlinear equation with $g(s) = s^{p}$, $p > 1$.
The weight function $a(x)$ is allowed to change its sign on the interval $I:= \mathopen[0,L\mathclose]$,
so that, according to a terminology which is standard in this setting (cf.~\cite{BerestyckiCapuzzo-DolcettaNirenberg2}),
we deal with a \textit{superlinear indefinite} problem.

Starting with the Eighties (see \cite{BandlePozioTesei,HessKato}), a great deal of attention has been devoted
to the study of nonlinear boundary value problems with a sign indefinite weight. The investigation of these problems
has its own interest from the point of view of the application of the methods of nonlinear analysis to
differential equations. A strong motivation also comes from the search of stationary solutions
of parabolic equations arising in different contexts, such as population dynamics and reaction-diffusion processes
(see, for instance, \cite{Ackermann} for a recent survey in this direction).
In such a situation, \eqref{two-pointBVP0} can be viewed as a one-dimensional version of the classical
Dirichlet problem
\begin{equation}\label{PDEBVP}
 \begin{cases}
\, -\Delta u=f(x,u) & \text{ in } \Omega   \\
\, u=0              & \text{ on } \partial\Omega,
\end{cases}
\end{equation}
which comes up, for instance, in the search of radially symmetric solutions of certain nonlinear PDEs
(see also Section~\ref{subsec-5.3}).

Existence or  multiplicity results of positive solutions for \eqref{PDEBVP} in the superlinear indefinite case
have been obtained in
\cite{AlamaTarantello, AmannLopezGomez, BerestyckiCapuzzo-DolcettaNirenberg1, BerestyckiCapuzzo-DolcettaNirenberg2}
(see also \cite{PapiniZanolin2} for a more complete list of references concerning different aspects related to the
study of superlinear indefinite problems, including the case of non-positive oscillating solutions).
Typically, the right-hand side of \eqref{PDEBVP} takes the form
\begin{equation*}
f(x,s) = \lambda s + a(x) g(s),
\end{equation*}
with $\lambda$ a real parameter.
In some cases also the weight function $a(x)$ depends on a parameter
which plays the role of strengthening or weakening the positive (or negative) part of the coefficient $a(x)$
(see \cite{BandlePozioTesei,LopezGomez}).
For convenience, in the present paper, when we need to underline such kind of dependence in the weight function,
we write
\begin{equation}\label{eq-1.2}
a(x) = a_{\mu}(x) := a^{+}(x) - \mu a^{-}(x), \quad \text{with } \, \mu > 0.
\end{equation}

In \cite{GaudenziHabetsZanolin2}, Gaudenzi, Habets and Zanolin proved the existence of at least three
positive solutions for the two-point boundary value problem associated to
\begin{equation}\label{eq-1.3}
u''+ a_{\mu}(x)u^{p}=0, \quad p>1,
\end{equation}
when $a_{\mu}(x)$ has two positive humps separated by a negative one, provided that $\mu > 0$ is
sufficiently large. The same multiplicity result has been obtained by Boscaggin in \cite{Boscaggin}
for the Neumann problem.
The technique of proof in \cite{GaudenziHabetsZanolin2}, based on the shooting method, has been
generalized in \cite{GaudenziHabetsZanolin4} in order to provide the existence of seven positive solutions
for \eqref{eq-1.3} (for the Dirichlet problem and with $\mu$ large) when $a(x)$ has three positive humps separated
by two negative ones. Generally speaking, this fact suggests the existence of $2^{n}-1$ positive solutions
(for $\mu$ large) when the weight function presents $n$ positive humps separated by $n-1$ negative ones.
It is interesting to observe that already in \cite{Gomez-RenascoLopez-Gomez}, for the one dimensional case
and
\begin{equation*}
f(x,s) = \lambda s + a(x) u^{p}, \quad p>1,
\end{equation*}
G\'{o}mez-Re\~{n}asco and L\'{o}pez-G\'{o}mez
conjectured this result (for $\lambda < 0$ sufficiently small) based on some numerical evidence (see also \cite{Ackermann}).

The same multiplicity result holds for the indefinite sublinear case, that is for equation
\begin{equation*}
u''+ a(x)u^{p}=0, \quad 0 < p < 1,
\end{equation*}
(cf.~\cite{BandlePozioTesei}).
In this situation, solutions may identically vanish on some sub-domains where $a(x)<0$, due to the lack of
Lipschitz character of the nonlinear term $g(s) = s^{p}$ at $s=0$.

More recently, Bonheure, Gomes and Habets in \cite{BonheureGomesHabets} extended
the multiplicity theorem in \cite{GaudenziHabetsZanolin4} to the PDEs
setting and they obtained a result about $2^{n}-1$ positive solutions for the problem
\begin{equation}\label{PDEBVPp}
\begin{cases}
\, -\Delta u= a(x)u^{p} & \text{ in } \Omega\\
\, u=0                  & \text{ on } \partial\Omega,
\end{cases}
\end{equation}
using a variational technique. In this context, $\Omega \subseteq {\mathbb{R}}^{N}$ is an open bounded domain
of class ${\mathcal{C}}^{1}$ and $1< p < (N+2)/(N-2)$ if $N \geq 3$.
Also for \eqref{PDEBVPp} the multiplicity result holds for $a(x) = a_{\mu}(x)$,
as in \eqref{eq-1.2}, with $\mu > 0$ sufficiently large.

On the other hand, if we have a positive weight function $a(x)$, it is known that the existence
of at least one positive solution to
\begin{equation}\label{PDEBVPag}
\begin{cases}
\, -\Delta u= a(x) g(u) & \text{ in } \Omega\\
\, u=0                  & \text{ on } \partial\Omega
\end{cases}
\end{equation}
is guaranteed for a general class of functions $g(s)$ (including $g(s)=s^{p}$, with $p > 1$, as a particular case).
Indeed, for $f(x,s) = a(x)g(s)$, the superlinear conditions at zero and at infinity can be generalized to suitable
hypotheses of crossing the first eigenvalue
(see \cite{AmbrosettiRabinowitz, Nussbaum1}, where results in this direction were
obtained using a variational and a topological approach, respectively). For instance, if $a(x)\equiv 1$, existence theorems of positive solutions
can be obtained, as in \cite{BrezisTurner, deFigueiredoLionsNussbaum}, provided that
\begin{equation*}
\limsup_{s\to 0^{+}} \dfrac{g(s)}{s} < \lambda_{1} < \liminf_{s\to +\infty} \dfrac{g(s)}{s}
\end{equation*}
(further technical assumptions on $g(s)$ and on the domain $\Omega$ can be required for $N > 1$).
In \cite[Chapter~3]{deFigueiredo}, de Figueiredo obtained sharp existence results
of positive solutions for
\begin{equation*}
\begin{cases}
\, -\Delta u= f(x,u) & \text{ in } \Omega\\
\, u=0               & \text{ on } \partial\Omega,
\end{cases}
\end{equation*}
by assuming
\begin{equation*}
f(x,0) \equiv 0, \quad \lim_{s\to 0^{+}} \dfrac{f(x,s)}{s} = m_{0}(x), \quad \lim_{s\to +\infty} \dfrac{f(x,s)}{s} = m_{\infty}(x).
\end{equation*}
The assumption of crossing the first eigenvalue is expressed by a hypothesis of the form $\mu_{1}(m_{0}) > 1 > \mu_{1}(m_{\infty})$,
where $\mu_{1}(m)$ is the first eigenvalue of $-\Delta u = \lambda m(x) u$. The different conditions at zero and at infinity
imply a change of the value of the fixed point index (for an associated operator) between small and large balls in the cone of positive functions
of a suitable Banach space $X$.
Hence the existence of a nontrivial fixed point is guaranteed by the nonzero index (or degree) on some open set $D \subseteq X$, with $0\notin D$.

In the ODEs case, namely for equation \eqref{two-pointBVP0}, various technical growth conditions on the nonlinearity
can be avoided. Existence theorems of positive solutions have been obtained in \cite{ErbeWang}
for the superlinear case and in
\cite{Ben-NaoumDeCoster, LanWebb, ManasevichNjokuZanolin, NjokuZanolin1}
for ``crossing the first eigenvalue'' type conditions. In this direction, an existence result has been produced in
\cite{GaudenziHabetsZanolin3} for \eqref{two-pointBVP0} with $f(x,s)$ as in \eqref{eq-1.1}.
In this case the weight function has nonconstant sign and may vanish on some subintervals of $I=\mathopen{[}0,L\mathclose{]}$.
It is also assumed that the set where $a(x) > 0$ is the union of $n$ pairwise disjoint intervals $I_{i}$.
Using an approach based on the theory of not well-ordered upper and lower-solutions,
the existence of a positive solution is guaranteed provided that
\begin{equation}\label{eq-1.5}
\limsup_{s\to 0^{+}} \dfrac{g(s)}{s} < \lambda_{0}
\quad \text{ and } \quad \liminf_{s\to +\infty} \dfrac{g(s)}{s} > \max_{i=1,\ldots,n} \lambda_{1}^{i},
\end{equation}
where $\lambda_{0}$ is the first eigenvalue of $\varphi'' + \lambda a^{+}(x) \varphi = 0$ on $I$ and
$\lambda_{1}^{i}$ is the first eigenvalue of $\varphi'' + \lambda a^{+}(x) \varphi = 0$ on $I_{i}$.

\medskip

A question, which naturally arises by a comparison between the above recalled existence theorem and
the multiplicity results in \cite{BonheureGomesHabets,GaudenziHabetsZanolin4}, concerns the possibility of
producing a theorem on the existence of $2^{n} -1$ positive solutions when $a(x) = a_{\mu}(x)$ is positive
on $n$ intervals separated by $n-1$ intervals of negativity and $g(s)$ satisfies a condition like \eqref{eq-1.5}.
The main goal of the present paper is to provide an affirmative answer to this question.
In this manner, we extend \cite{GaudenziHabetsZanolin3} and \cite{GaudenziHabetsZanolin4} at the same
time and, moreover, we are able to prove that the multiplicity results in \cite{GaudenziHabetsZanolin4}
hold for a broad class of nonlinearities which include $g(s) = s^{p}$, $p >1$, as a special case.
More precisely, the following result holds.

\begin{theorem}\label{th-1.1}
Suppose that $a \colon \mathopen{[}0,L\mathclose{]}\to {\mathbb{R}}$ is a continuous function and there are
$2n$ points
\begin{equation*}
0 = \sigma_{1} < \tau_{1} < \sigma_{2} < \tau_{2} < \ldots < \sigma_{n} < \tau_{n} = L,
\end{equation*}
such that $\tau_{1},\sigma_{2},\tau_{2},\ldots,\sigma_{n}$ are simple zeros of $a(x)$ and, moreover,
\begin{itemize}
\item[$\bullet\;$] $a(x)\geq0$, for all $x\in I_{i}:=\mathopen[\sigma_{i},\tau_{i}\mathclose]$, $i=1,\ldots,n$;
\item[$\bullet\;$] $a(x)\leq0$, for all $x\in\mathopen[\tau_{i},\sigma_{i+1}\mathclose]$, $i=1,\ldots,n-1$.
\end{itemize}
Assume also that $g\colon{\mathbb{R}}^{+} \to {\mathbb{R}}^{+}$ is a continuous function with $g(0) = 0$ and $g(s) > 0$ for $s > 0$
and, moreover, \eqref{eq-1.5} holds.
Then there exists $\mu^{*}>0$ such that, for each $\mu > \mu^{*}$,
problem
\begin{equation}\label{two-pointBVPmu0}
\begin{cases}
\, u''+a_{\mu}(x)g(u)=0 \\
\, u(0)=u(L)=0
\end{cases}
\end{equation}
has at least $2^{n} -1$ positive solutions.
\end{theorem}

\noindent
The assumptions on the sign of $a(x)$ do not prevent the possibility that $a(x)$ is identically zero on some
subintervals of $\mathopen{[}\sigma_{i},\tau_{i}\mathclose{]}$ or $\mathopen{[}\tau_{i},\sigma_{i+1}\mathclose{]}$.
The hypothesis that $\tau_{1},\sigma_{2},\tau_{2},\ldots,\sigma_{n}$ are simple zeros of $a(x)$
is considered here only in order to provide a simpler statement of our theorem. Indeed, this condition
can be significantly relaxed (cf.~Theorem~\ref{th-5.3}).

Figure~\ref{fig-01} adds a graphical explanation to Theorem~\ref{th-1.1}
in a case in which the weight function has two positive humps separated by a negative one.
We stress that in our example $g(s)/s\not\to +\infty$ as $s\to +\infty$, so that it shows a
case of applicability of Theorem~\ref{th-1.1} which is not contained in \cite{BonheureGomesHabets,GaudenziHabetsZanolin2,GaudenziHabetsZanolin4}.

\begin{figure}[h!]
\centering
\includegraphics[width=0.4\textwidth]{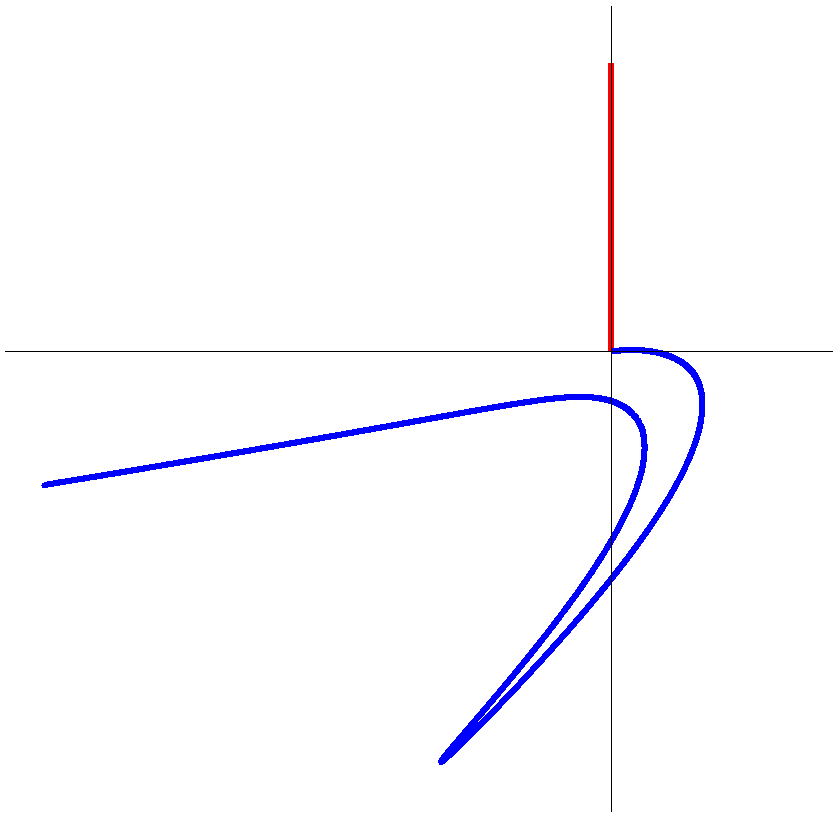}\qquad
\includegraphics[width=0.5\textwidth]{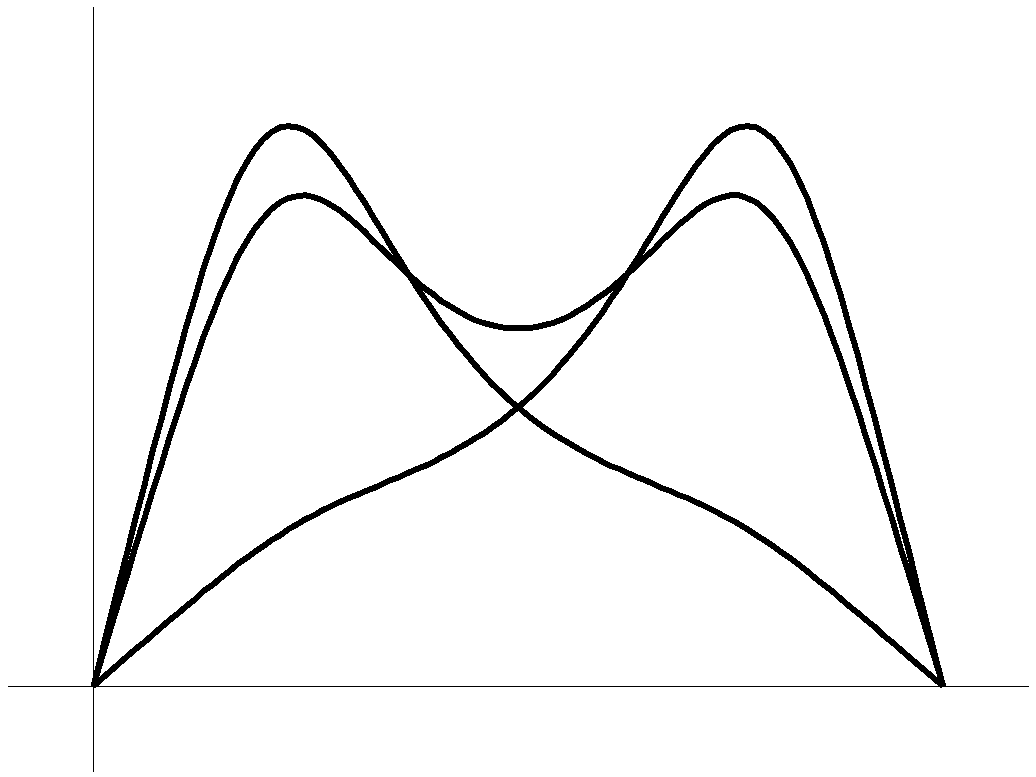}
\caption{\small{An example of three positive solutions for problem \eqref{two-pointBVPmu0}.
For this numerical simulation we have chosen $I = \mathopen{[}0,1\mathclose{]}$, $a(x) = \sin(3\pi x)$, $\mu = 0.5$
and $g(s) = \max\{0,100\, s\arctan|s|\}$. On the left we have shown the image of
the segment $\{0\}\times \mathopen[0,5\mathclose]$ through the Poincar\'{e} map in the phase-plane $(u,u')$.
It intersects the negative part of the $u'$-axis in three points. This means that there are three positive
initial slopes from which depart three solutions at $x=0$ which are positive on $\mathopen{]}0,1\mathclose{[}$
and vanish at $x=1$. On the right we have shown the graphs of these three solutions.}}
\label{fig-01}
\end{figure}

Theorem~\ref{th-1.1} follows from a more general result concerning \eqref{two-pointBVP0} in which we impose
nonuniform conditions on $f(x,s)/s$ (this is also in the spirit of some previous works
\cite{deFigueiredoGossezUbilla,ObersnelOmari}, where ``local'' superlinearity conditions were considered).
Even if we have confined most of our applications to the case
\begin{equation*}
f(x,s) = a(x) g(s),
\end{equation*}
we observe that our general results can be applied to other kind of nonlinearities, as
\begin{equation*}
f(x,s) =\sum_{i=1}^{k} a_{i}(x)g_{i}(s).
\end{equation*}

To prove our results we use a different approach with respect to \cite{GaudenziHabetsZanolin4}
and \cite{BonheureGomesHabets}, where a shooting method and a variational technique were employed.
Indeed, our proofs, based on topological degree, are more in the frame of the classical approach for the
search of fixed points based on the fixed point index for positive operators. Using the additivity/excision
property of the Leray-Schauder degree, we localize nontrivial fixed points on suitable open domains
of the Banach space ${\mathcal{C}}(I)$ of the continuous functions defined on $I$. In general,
our open domains are unbounded and therefore we apply an extension of the degree theory
for locally compact operators (cf.~\cite{Nussbaum2, Nussbaum3}). A typical open (unbounded) set that we
introduce along the proof of our multiplicity results for \eqref{two-pointBVPmu0} is made by functions which
are bounded by suitable constants \textit{only} on the subintervals of positivity $I_{i}$. By the convexity of the solutions
on the intervals of negativity for the weight, this is enough to have a full control of
the solutions (and hence a priori bounds) on the whole domain $I$.
As usual, working with topological degree, the main efforts are concentrated in showing that solutions do not
appear on the boundary along some homotopies. This is guaranteed by suitable technical estimates.

An advantage of an approach based on topological degree consists in the fact that when the degree is well defined and nonzero,
then existence of solutions persists also for small perturbations of the operator. In this manner, our multiplicity result
for problem \eqref{two-pointBVPmu0} extends to the equation
\begin{equation*}
u''+ \lambda u + a_{\mu}(x)g(u)=0,
\end{equation*}
provided that $|\lambda|$ is small enough. Therefore we can establish a complete proof of
G\'{o}mez-Re\~{n}asco and L\'{o}pez-G\'{o}mez conjecture in all its aspects.

\medskip

The plan of the paper is the following.
In Section~\ref{section2} we introduce the key ingredients.
More in detail we illustrate the problem, we define an equivalent fixed point problem
and we list the hypotheses we are going to assume on $f(x,s)$.
Moreover we prove some preliminary technical lemmas that permits to compute the topological degree on suitable small and large balls.
Once nontrivial solutions are obtained,  we also have positive solutions. As usual, this follows from
a maximum principle (Lemma~\ref{Maximum principle}).

In Section~\ref{section3} we present existence theorems.
Using the lemmas of the previous section, we prove that there exists at least a positive solution for \eqref{two-pointBVP0}.
The results we offer are more general versions of \cite[Theorem~4.1]{GaudenziHabetsZanolin3} and \cite[Corollary~4.2]{GaudenziHabetsZanolin3}.
In Theorem~\ref{Th+LinCond} we weaken the hypothesis on the growth of $f(x,s)/s$ at infinity,
by assuming the linear growth of $f(x,s)$ in the subintervals where the growth condition is not valid.

Section~\ref{section4} is devoted to the proof of our main result,
namely Theorem~\ref{Th-Multiplicity}, which deals with the multiplicity of solutions.
By employing the Leray-Schauder topological degree, we deduce the existence of $2^{n}-1$ positive solutions of our BVP.

In Section~\ref{section5} we analyze boundary value problems with $f(x,s)=a(x)g(s)$, as a special case.
We discuss the results obtained in the previous sections in this particular context.
In this way we obtain the existence and multiplicity theorems we look for.
We finish that section with some remarks concerning radially symmetric solutions for \eqref{PDEBVPag}
when $a(x) = {\mathcal{A}}(\|x\|)$ and $\Omega$ is an annular domain.
Possible generalizations on the weight function $a(x)$ are considered, too.

We conclude our paper with Appendix~\ref{section6}
where we recall the definition and a few relevant properties of the Leray-Schauder degree on possibly unbounded sets.
Therein we present two theorems that help us to compute the topological degree in the first sections.

\section{Preliminary results}\label{section2}

In this section we collect some classical and basic facts which are then applied in the proofs of our main results.

Let $I\subseteq {\mathbb{R}}$ be a nontrivial compact interval. Set ${\mathbb{R}}^{+}:=\mathopen[0,+\infty\mathclose[$
and let $f\colon I\times {\mathbb{R}}^{+}\to {\mathbb{R}}$ be an $L^{1}$-Carath\'{e}odory function, that is
\begin{itemize}
  \item $x\mapsto f(x,s)$ is measurable for each $s\in{\mathbb{R}}^{+}$;
  \item $s\mapsto f(x,s)$ is continuous for a.e.~$x\in I$;
  \item for each $r>0$ there is $\gamma_{r}\in L^{1}(I,{\mathbb{R}}^{+})$ such that $|f(x,s)|\leq\gamma_{r}(x)$,
    for a.e.~$x\in I$ and for all $|s|\leq r$.
\end{itemize}

Without loss of generality, we suppose $I:=\mathopen[0,L\mathclose]$ (different choices of $I$ can be made).
We study the two-point boundary value problem
\begin{equation}\label{two-pointBVP}
\begin{cases}
\, u''+f(x,u)=0 \\
\, u(0)=u(L)=0.
\end{cases}
\end{equation}
A \textit{solution} of \eqref{two-pointBVP} is an absolutely continuous function $u\colon\mathopen[0,L\mathclose]\to {\mathbb{R}}^{+}$
such that its derivative $u'(x)$ is absolutely continuous and $u(x)$ satisfies \eqref{two-pointBVP}
for a.e.~$x\in\mathopen[0,L\mathclose]$.
We look for \textit{positive solutions} of \eqref{two-pointBVP},
that is solutions $u$ such that $u(x)>0$ for every $x\in
\mathopen]0,L\mathclose[$.

Throughout the paper we suppose
\begin{equation*}
f(x,0) = 0, \quad \text{for a.e. } x\in I. \leqno{(f^{*})}
\end{equation*}
Using a standard procedure, we extend
$f(x,s)$ to a function $\tilde{f}\colon
I\times{\mathbb{R}}\to{\mathbb{R}}$ defined as
\begin{equation*}
\tilde{f}(x,s)=
\begin{cases}
\, f(x,s), & \text{if } s\geq0; \\
\, 0,  & \text{if } s\leq0;
\end{cases}
\end{equation*}
and we study the modified boundary value problem
\begin{equation}\label{BVP}
\begin{cases}
\, u''+\tilde{f}(x,u)=0 \\
\, u(0)=u(L)=0.
\end{cases}
\end{equation}
As is well known (by the maximum principle), all the possible solutions of \eqref{BVP} are non-negative and
hence solutions of \eqref{two-pointBVP}. Actually, we need the following lemma concerning the positivity of the solutions
of the BVP
\begin{equation}\label{mpBVP}
\begin{cases}
\, v''+h(x,v)=0 \\
\, v(0)=v(L)=0.
\end{cases}
\end{equation}
In the applications, we have $h=\tilde{f}$ or $h$ a suitably chosen function greater than $\tilde{f}$.

\begin{lemma}\label{Maximum principle}
Let $h\colon I\times {\mathbb{R}}\to {\mathbb{R}}$ be an $L^{1}$-Carath\'{e}odory function.
\begin{itemize}
\item [$(i)$] If
    \begin{equation*}
    h(x,s)\geq0, \quad \text{a.e. } x\in I, \text{ for all } s\leq0,
    \end{equation*}
    then any solution of \eqref{mpBVP} is non-negative on $I$.
\item [$(ii)$]
    If
    $h(x,0)\equiv0$ and
    there exist $k_{1},k_{2}\in L^{1}(I,{\mathbb{R}}^{+})$ such that
    \begin{equation*}
    \begin{aligned}
    & \liminf_{s\to0^{+}}\dfrac{h(x,s)}{s}\geq -k_{1}(x), \quad \text{uniformly a.e. } x\in I; \\
    & \limsup_{s\to0^{+}}\dfrac{h(x,s)}{s}\leq k_{2}(x), \quad \text{uniformly a.e. } x\in I;
    \end{aligned}
    \end{equation*}
    then every non-trivial non-negative solution $v$ of \eqref{mpBVP} satisfies $v(x)>0$, for all $x\in\mathopen]0,L\mathclose[$
    and, moreover, $v'(0) > 0 > v'(L)$.
\end{itemize}
\end{lemma}
The standard proof is omitted (see, for instance, \cite{ManasevichNjokuZanolin}).

We remark that Lemma~\ref{Maximum principle} is stated in a form which is useful for our applications.
For example, assertion $(ii)$ can be equivalently expressed in a simpler manner:
in effect if we only suppose that
there exists $k\in L^{1}(I,{\mathbb{R}}^{+})$ such that
\begin{equation*}
\limsup_{s\to0^{+}}\dfrac{|h(x,s)|}{s}\leq k(x), \quad \text{uniformly a.e. } x\in I,
\end{equation*}
the conclusion remains valid.

In view of Lemma~\ref{Maximum principle}, from now on, we suppose
\begin{itemize}
\item[$(f_{0}^{-})$]
there exists a function $q_{-}\in L^{1}(I,{\mathbb{R}}^{+})$ such that
\begin{equation*}
\liminf_{s\to0^{+}}\dfrac{f(x,s)}{s}\geq -q_{-}(x), \quad \text{uniformly a.e. } x\in I.
\end{equation*}
\end{itemize}

For the sequel we also need to introduce a suitable notation concerning the first eigenvalue of a linear problem with
non-negative weight.

Let $J:=\mathopen[x_{1},x_{2}\mathclose]\subseteq I$ be a compact subinterval and $q\in L^{1}(J,{\mathbb{R}}^{+})$
with $q\not\equiv 0$, namely $q>0$ a.e.~on a set of positive measure.
We denote by $\mu_{1}^{J}(q)$ the first (positive) eigenvalue of
\begin{equation*}
\begin{cases}
\, \varphi''+\mu q(x)\varphi=0 \\
\, \varphi|_{\partial J}=0.
\end{cases}
\end{equation*}
Sturm theory (in the Carath\'{e}odory setting) guarantees that $\mu_{1}^{J}(q)$ is a simple eigenvalue with an associated eigenfunction
$\varphi$ which
is positive in the interior of $J$ and such that $\varphi'(x_{1}) > 0 > \varphi'(x_{2})$.
As a notational convention, when $J = I = \mathopen[0,L\mathclose]$,
we denote the first eigenvalue simply by $\mu_{1}(q)$.

\medskip

As mentioned in the Introduction, our approach to the search of positive solutions of \eqref{two-pointBVP} is based on
Leray-Schauder topological degree. Accordingly, we transform problem \eqref{BVP} into an equivalent fixed point problem for
an associated operator which is the classical one defined by means of the Green function
for the operator $u\mapsto - u''$ with the two-point boundary conditions. Namely, we define
$\Phi\colon\mathcal{C}(I)\to\mathcal{C}(I)$ by
\begin{equation}\label{operator}
(\Phi u)(x):= \int_{I} G(x,\xi)\tilde{f}(\xi,u(\xi)) ~\!d\xi.
\end{equation}
The operator $\Phi$ is completely continuous in $\mathcal{C}(I)$ endowed with the $\sup$-norm $\|\cdot\|_{\infty}$.

\medskip

Our goal is to find multiple fixed points for $\Phi$ using a degree theoretic approach. To this aim, we present now some
technical lemmas which are stated in a form that is suitable to be subsequently applied for the computation of the
topological degree via homotopy procedures.

\begin{lemma}\label{lemmar0}
Let $f(x,s)$ satisfy $(f^{*})$ and $(f_{0}^{-})$. Suppose
\begin{itemize}
\item[$(f_{0}^{+})$]
there exists a measurable function
$q_{0}\in L^{1}(I,{\mathbb{R}}^{+})$ with $q_{0}\not\equiv0$,
such that
\begin{equation*}
\limsup_{s\to 0^{+}}\dfrac{f(x,s)}{s}\leq q_{0}(x), \quad \text{uniformly a.e. } x\in I,
\end{equation*}
and
\begin{equation*}
\mu_{1}(q_{0})>1.
\end{equation*}
\end{itemize}
Then there exists $r_{0}>0$ such that
every solution $u(x)\geq 0$ of the two-point BVP
\begin{equation}\label{BVP1}
\begin{cases}
\, u''+ \vartheta f(x,u)=0, \quad 0\leq\vartheta\leq1, \\
\, u(0)=u(L)=0
\end{cases}
\end{equation}
satisfying $\max_{x\in I}u(x)\leq r_{0}$ is such that $u(x) = 0$, for all $x\in I$.
\end{lemma}

\begin{proof}
Let $\varepsilon>0$ be such that
\begin{equation*}
\hat{\mu}:=\mu_{1}(q_{0}+\varepsilon)>1
\end{equation*}
(see \cite{DalbonoZanolin},  \cite[p.~44]{deFigueiredo} or \cite{Zettl}).
By condition $(f_{0}^{+})$, there exists $r_{0}>0$ such that
\begin{equation*}
\dfrac{f(x,s)}{s}\leq q_{0}(x)+\varepsilon, \quad \text{a.e. } x\in I, \; \forall \, 0<s\leq r_{0}.
\end{equation*}

Let $\varphi$ be a positive eigenfunction of
\begin{equation*}
\begin{cases}
\, \varphi''+ \hat{\mu}[q_{0}(x)+\varepsilon]\varphi=0 \\
\, \varphi(0)= \varphi(L)=0.
\end{cases}
\end{equation*}
Then $\varphi(x)>0$, $\forall \, x\in \mathopen]0,L\mathclose[$, and $\varphi'(0)>0>\varphi'(L)$.

In order to prove the statement of our lemma, suppose, by contradiction, that there exist $\vartheta\in\mathopen[0,1\mathclose]$
and a solution $u(x)\geq 0$
of \eqref{BVP1} such that $\max_{x\in I} u(x) = r$ for some $0 < r \leq r_{0}$.
Notice that, by the choice of $r_{0}$, we have that
\begin{equation*}
\vartheta f(x,u(x)) \leq (q_{0}(x)+\varepsilon)u(x), \quad \text{a.e. } x\in I.
\end{equation*}

Using a Sturm comparison argument, we obtain
\begin{equation*}
\begin{aligned}
0  & = \bigl{[}u'(x)\varphi(x)-u(x)\varphi'(x)\bigr{]}_{x=0}^{x=L}
\\ & = \int_{0}^{L}\dfrac{d}{dx}\Bigl{[}u'(x)\varphi(x)-u(x)\varphi'(x)\Bigr{]} ~\!dx
\\ & = \int_{0}^{L}\Bigl{[}u''(x)\varphi(x)-u(x)\varphi''(x)\Bigr{]} ~\!dx
\\ & = \int_{0}^{L}\Bigl{[}-\vartheta f(x,u(x))\varphi(x)+u(x)\hat{\mu}(q_{0}(x)+\varepsilon)\varphi(x)\Bigr{]} ~\!dx
\\ & = \int_{0}^{L}\Bigl[(q_{0}(x)+\varepsilon)u(x)-\vartheta f(x,u(x))\Bigr]\varphi(x) ~\!dx
\\ &   \quad + (\hat{\mu} -1)\int_{0}^{L}(q_{0}(x)+\varepsilon)u(x)\varphi(x) ~\!dx
\\ & > 0,
\end{aligned}
\end{equation*}
a contradiction.
\end{proof}

A direct application of Lemma~\ref{lemmar0} permits to compute the degree on small neighborhoods of the origin.
Indeed, we have

\begin{lemma}\label{lemmar}
Let $f(x,s)$ satisfy $(f^{*})$, $(f_{0}^{-})$ and $(f_{0}^{+})$.
Then there exists $r_{0}>0$ such that
\begin{equation*}
deg(Id-\Phi,B(0,r),0)=1, \quad \forall \, 0<r\leq r_{0}.
\end{equation*}
\end{lemma}

\begin{proof}
Let $r_{0}$ be as in Lemma~\ref{lemmar0} and let us fix $r\in\mathopen]0,r_{0}\mathclose]$. If $u \in \mathcal{C}(I)$ satisfies
$u = \vartheta\Phi(u)$, for some  $0\leq \vartheta \leq 1$,
then $u$ is a solution of the equation $u'' + \vartheta\tilde{f}(x,u) =0$ with $u(0) = u(L) = 0$.
Now, either $u=0$, or (according to Lemma~\ref{Maximum principle}) $u(x) > 0$ for each $x\in\mathopen]0,L\mathclose[$
and therefore  $u$ is a solution of \eqref{BVP1}. Hence, Lemma~\ref{lemmar0} and the choice of $r$
imply that $\|u\|_{\infty} \neq r$. This, in turn, implies that
\begin{equation*}
u\neq \vartheta \Phi(u), \quad \forall \, \vartheta\in\mathopen[0,1\mathclose], \; \forall \, u\in \partial B(0,r).
\end{equation*}
By the homotopic invariance property of the topological degree, we conclude that
$deg(Id-\Phi,B(0,r),0)=deg(Id,B(0,r),0)=1$.
\end{proof}

\medskip

As a next step, we give a result which will be used to compute the degree on large balls. It follows from Lemma~\ref{lemmaRJ} below,
where we assume suitable conditions on $f(x,s)/s$ for $s > 0$ and large. Notice that our assumptions are ``local'' (in the spirit
of \cite{deFigueiredoGossezUbilla} and \cite{ObersnelOmari}), in the sense that we do not require their validity on the whole domain.

\begin{lemma}\label{lemmaRJ}
Suppose there exists a closed interval $J \subseteq I$ such that
\begin{equation*}
f(x,s)\geq 0, \quad \text{a.e. } x\in J, \; \forall \, s\geq0;
\end{equation*}
and there is a measurable function $q_{\infty}\in L^{1}(J,{\mathbb{R}}^{+})$ with $q_{\infty}\not\equiv0$, such that
\begin{equation}\label{condinfty}
\liminf_{s\to+\infty}\dfrac{f(x,s)}{s}\geq q_{\infty}(x), \quad \text{uniformly a.e. } x\in J.
\end{equation}
Suppose
\begin{equation*}
\mu_{1}^{J}(q_{\infty})<1.
\end{equation*}
Then there exists $R_{J}>0$ such that
for each Carath\'{e}odory function $g\colon I \times{\mathbb{R}}^{+}\to {\mathbb{R}}$ with
\begin{equation*}
g(x,s)\geq f(x,s), \quad \text{a.e. } x\in J, \; \forall \, s\geq0,
\end{equation*}
every solution $u(x)\geq0$ of the two-point BVP
\begin{equation}\label{BVP2}
\begin{cases}
\, u''+g(x,u)=0 \\
\, u(0)=u(L)=0
\end{cases}
\end{equation}
satisfies $\max_{x\in J}u(x)<R_{J}$.
\end{lemma}
We stress that the constant $R_{J}$ does not depend on the function $g(x,s)$.
\begin{proof}
Just to fix a notation along the proof, we set $J:=\mathopen[x_{1},x_{2}\mathclose]$.
By contradiction, suppose there is not a constant $R_{J}$ with those properties. So, for all $n>0$
there exists $\tilde{u}_{n}\geq0$ solution of \eqref{BVP2} with $\max_{x\in J}\tilde{u}_{n}(x)=:\hat{R}_{n}>n$.

Let $q_{n}(x)$ be a monotone nondecreasing sequence of non-negative measurable functions such that
\begin{equation*}
f(x,s) \geq q_{n}(x) s, \quad \text{a.e. } x\in J, \; \forall \, s \geq n,
\end{equation*}
and $q_{n}\to q_{\infty}$ uniformly almost everywhere in $J$.
The existence of such a sequence comes from condition \eqref{condinfty}.

Fix $\varepsilon<(1-\mu_{1}^{J}(q_{\infty}))/2$. Hence, there exists an integer $N >0$ such that $q_{n}\not\equiv 0$ for each $n\geq N$ and, moreover,
\begin{equation*}
\nu_{n}:=\mu_{1}^{J}(q_{n}) \leq 1-\varepsilon , \quad \forall \, n\geq N.
\end{equation*}
Now we fix $N$ as above and denote by
$\varphi$ the positive eigenfunction of
\begin{equation*}
\begin{cases}
\, \varphi''+ \nu_{N} q_{N}(x)\varphi=0 \\
\, \varphi(x_{1})=\varphi(x_{2})=0 ,
\end{cases}
\end{equation*}
with $\|\varphi\|_{\infty}=1$. Then $\varphi(x)>0$, $\forall \, x\in\mathopen]x_{1},x_{2}\mathclose[$,
and $\varphi'(x_{1})>0>\varphi'(x_{2})$.

For each $n\geq N$, let $J'_{n}\subseteq J$ be the maximal closed interval, such that
\begin{equation*}
\tilde{u}_{n}(x)\geq N, \quad \forall \, x\in J'_{n}.
\end{equation*}
By the concavity of the solution in the interval $J$ and the definition of $J'_{n}$,
we also have that
\begin{equation*}
 \tilde{u}_{n}(x)\leq N, \quad \forall \, x\in J\setminus J'_{n}.
\end{equation*}
Another consequence of the concavity of $\tilde{u}_{n}$ on $J$ ensures that
\begin{equation*}
\tilde{u}_{n}(x)\geq \dfrac{\hat{R}_{n}}{x_{2}-x_{1}}\min\{x-x_{1},x_{2}-x\}, \quad \forall x\in J,
\end{equation*}
(see \cite[p.~420]{GaudenziHabetsZanolin3} for a similar estimate).
Hence, if we take $n\geq 2N$, we find that $\tilde{u}_{n}(x)\geq N$, for all $x$ in the well-defined closed interval
\begin{equation*}
A_{n}:=\biggl[x_{1}+\dfrac{N}{\hat{R}_{n}}(x_{2}-x_{1}),x_{2}-\dfrac{N}{\hat{R}_{n}}(x_{2}-x_{1})\biggr]\subseteq J'_{n}.
\end{equation*}
By construction, ${\it meas}(J\setminus J'_{n})\leq {\it meas}(J\setminus A_{n})\to 0$ as $n\to \infty$.

Using a Sturm comparison argument, for each $n\geq N$, we obtain
\begin{equation*}
\begin{aligned}
0  & \geq \tilde{u}_{n}(x_{2})\varphi'(x_{2})-\tilde{u}_{n}(x_{1})\varphi'(x_{1})
 = \Bigl{[}\tilde{u}_{n}(x)\varphi'(x)-\tilde{u}'_{n}(x)\varphi(x)\Bigr{]}_{x=x_{1}}^{x=x_{2}}
\\ & = \int_{x_{1}}^{x_{2}}\dfrac{d}{dx}\Bigl{[}\tilde{u}_{n}(x)\varphi'(x)-\tilde{u}'_{n}(x)\varphi(x)\Bigr{]} ~\!dx
\\ & = \int_{J}\Bigl{[}\tilde{u}_{n}(x)\varphi''(x)-\tilde{u}''_{n}(x)\varphi(x)\Bigr{]} ~\!dx
\\ & = \int_{J}\Bigl{[}-\tilde{u}_{n}(x)\nu_{N}q_{N}(x)\varphi(x)+g(x,\tilde{u}_{n}(x))\varphi(x)\Bigr{]} ~\!dx
\\ & =  \int_{J}\Bigl{[}g(x,\tilde{u}_{n}(x))-\nu_{N}q_{N}(x)\tilde{u}_{n}(x)\Bigr{]}\varphi(x) ~\!dx
\\ & \geq \int_{J}\Bigl{[}f(x,\tilde{u}_{n}(x))-\nu_{N}q_{N}(x)\tilde{u}_{n}(x)\Bigr{]}\varphi(x) ~\!dx
\\ & = \int_{J'_{n}}\Bigl{[}f(x,\tilde{u}_{n}(x))-q_{N}(x)\tilde{u}_{n}(x)\Bigr{]}\varphi(x) ~\!dx +
(1-\nu_{N})\int_{J'_{n}}q_{N}(x)\tilde{u}_{n}(x)\varphi(x) ~\!dx
\\ & \quad + \int_{J\setminus J'_{n}}\Bigl{[}f(x,\tilde{u}_{n}(x))-\nu_{N}q_{N}(x)\tilde{u}_{n}(x)\Bigr{]}\varphi(x) ~\!dx.
\end{aligned}
\end{equation*}
Recalling that
\begin{equation*}
f(x,s)\geq q_{N}(x) s, \quad \text{a.e. } x\in J, \; \forall \, s\geq N,
\end{equation*}
we know that
\begin{equation*}
f(x,\tilde{u}_{n}(x))-q_{N}(x)\tilde{u}_{n}(x)\geq 0, \quad \text{a.e. } x\in J'_{n}, \; \forall \, n\geq N.
\end{equation*}
Then, using the Carath\'{e}odory assumption, which implies that
\begin{equation*}
|f(x,s)|\leq \gamma_{N}(x), \quad \text{a.e. } x\in J, \; \forall \, 0\leq s \leq N,
\end{equation*}
where $\gamma_{N}$ is a suitably non-negative integrable function, we obtain
\begin{equation*}
\begin{aligned}
0  & \geq
\int_{J'_{n}}\Bigl{[}f(x,\tilde{u}_{n}(x))-q_{N}(x)\tilde{u}_{n}(x)\Bigr{]}\varphi(x) ~\!dx +
(1-\nu_{N})\int_{J'_{n}}q_{N}(x)\tilde{u}_{n}(x)\varphi(x) ~\!dx
\\ & \quad + \int_{J\setminus J'_{n}}\Bigl{[}f(x,\tilde{u}_{n}(x))-\nu_{N}q_{N}(x)\tilde{u}_{n}(x)\Bigr{]}\varphi(x) ~\!dx
\\ & \geq
\varepsilon N \int_{J'_{n}}q_{N}(x)\varphi(x) ~\!dx + \int_{J\setminus J'_{n}}\Bigl{[}-\gamma_{N}(x)-N\nu_{N}q_{N}(x)\Bigr{]} ~\!dx
\\ & =
\varepsilon N \int_{J}q_{N}(x)\varphi(x) ~\!dx -  \varepsilon N\int_{J\setminus J'_{n}} q_{N}(x)\varphi(x) ~\!dx
\\ & \quad - \int_{J\setminus J'_{n}}\Bigl{[}\gamma_{N}(x) + N\nu_{N}q_{N}(x)\Bigr{]} ~\!dx.
\end{aligned}
\end{equation*}
Passing to the limit as $n\to \infty$ and using the dominated convergence theorem, we obtain
\begin{equation*}
0 \geq \varepsilon N \int_{J}q_{N}(x)\varphi(x) ~\!dx > 0,
\end{equation*}
a contradiction.
\end{proof}

An immediate consequence of Lemma~\ref{lemmaRJ} is the following result (Lemma~\ref{lemmaR}), where we assume a
specific sign condition on the function $f(x,s)$. Such condition will play an important role in all our applications.

\begin{itemize}
\item[$(H)$] \textit{
Suppose that there exist $n\geq 1$ intervals $I_{1},\ldots,I_{n}$, closed and pairwise disjoint, such that
\begin{equation*}
\begin{aligned}
   & f(x,s) \geq 0, \quad \mbox{for a.e. }  x\in \bigcup_{i=1}^{n} I_{i} \mbox{ and for all } s\geq 0;
\\ & f(x,s) \leq 0, \quad \mbox{for a.e. }  x\in I\setminus \bigcup_{i=1}^{n} I_{i} \mbox{ and for all } s\geq 0.
\end{aligned}
\end{equation*}
}
\end{itemize}
If $n=1$, condition $(H)$ simply requires that there exists a compact subinterval $J\subseteq I$ such that for each $s\geq 0$ it holds that
$f(x,s)\geq 0$ for a.e.~$x\in J$ and $f(x,s) \leq 0$ for a.e.~$x\in I\setminus J$ (the possibility that $J = I$ is not excluded).

\begin{lemma}\label{lemmaR}
Let $f(x,s)$ satisfy $(H)$ as well as $(f^{*})$ and $(f_{0}^{-})$. Assume also
\begin{itemize}
\item[$(f_{\infty})$] for all $i=1,\ldots,n$ there exists a measurable function $q_{\infty}^{i}\in L^{1}(I_{i},{\mathbb{R}}^{+})$
with $q_{\infty}^{i}\not\equiv0$, such that
\begin{equation*}
\liminf_{s\to+\infty}\dfrac{f(x,s)}{s}\geq q_{\infty}^{i}(x), \quad \text{uniformly a.e. } x\in I_{i},
\end{equation*}
and
\begin{equation*}
\mu_{1}^{I_{i}}(q_{\infty}^{i})<1.
\end{equation*}
\end{itemize}
Then there exists $R^{*}>0$ such that
\begin{equation*}
deg(Id-\Phi,B(0,R),0)=0, \quad \forall \, R\geq R^{*}.
\end{equation*}
\end{lemma}

\begin{proof}
Define
\begin{equation*}
R^{*}:=\max_{i=1,\ldots,n}R_{I_{i}}>0,
\end{equation*}
with $R_{I_{i}}>0$ defined as in Lemma~\ref{lemmaRJ}. Let us also fix a radius $R\geq R^{*}$.

We denote by $\mathbbm{1}_{A}$ the characteristic function of the set $A:= \bigcup_{i=1}^{n} I_{i}$.
We set $v(x):=\int_{I}G(x,s)\mathbbm{1}_{A}(s) ~\!ds$ and we consider in $\mathcal{C}(I)$ the operator equation
\begin{equation*}
u =\Phi(u) + \alpha v,\quad \mbox{for } \alpha \geq 0.
\end{equation*}
Clearly, any nontrivial solution $u$ of the above equation is a solution of
$u'' + \tilde{f}(x,u) + \alpha \mathbbm{1}_{A}(x) = 0$ with $u(0) = u(L) = 0$. By the first part of Lemma~\ref{Maximum principle}
we know that $u(x)\geq 0$ for all $x\in I$. Hence, $u$ is a non-negative solution of
\eqref{BVP2} with
\begin{equation*}
g(x,s) = f(x,s) + \alpha \mathbbm{1}_{A}(x),
\end{equation*}
for $\alpha \geq 0$.
By definition, we have that $g(x,s) \geq f(x,s)$ for a.e.~$x\in A$ and for all $s\geq 0$, and also $g(x,s) = f(x,s) \leq 0$
for a.e.~$x\in I\setminus A$ and for all $s\geq 0$.
By the convexity of the solutions of \eqref{BVP2} in the intervals of $I\setminus A$, we obtain
\begin{equation*}
\max_{x\in I}u(x)=\max_{x\in A}u(x)
\end{equation*}
and, as an application of Lemma~\ref{lemmaRJ} on each of the intervals $I_{i}$, we conclude that
\begin{equation*}
\|u\|_{\infty} < R^{*} \leq R.
\end{equation*}
As a consequence,
\begin{equation*}
u \neq \Phi(u) + \alpha v,\quad \mbox{for all } u\in \partial B(0,R) \mbox{ and } \alpha \geq 0,
\end{equation*}
and thus the thesis follows from the second part of Theorem~\ref{deFigueiredo}.
\end{proof}

\section{Some existence results}\label{section3}

An immediate consequence of Lemma~\ref{lemmar} and Lemma~\ref{lemmaR} is the following existence result which
generalizes \cite[Theorem~4.1]{GaudenziHabetsZanolin3}.

\begin{theorem}\label{Theorem4.1GHZ}
Let $f(x,s)$ satisfy $(f^{*})$, $(f_{0}^{-})$, $(f_{0}^{+})$ and $(H)$ with $(f_{\infty})$.
Then there exists at least a positive solution of the two-point BVP \eqref{two-pointBVP}.
\end{theorem}

\begin{proof}
Let us take $r_{0}$ as in Lemma~\ref{lemmar} and $R^{*}$ as in Lemma~\ref{lemmaR}. Clearly $0<r_{0}<R^{*}<+\infty$. Then
\begin{equation*}
\begin{aligned}
& deg(Id-\Phi,B(0,R^{*})\setminus B[0,r_{0}],0)=
\\ & = deg(Id-\Phi,B(0,R^{*}),0)-deg(Id-\Phi,B(0,r_{0}),0)=
\\ & =0-1=-1\neq0.
\end{aligned}
\end{equation*}
Hence a nontrivial solution exists and the claim follows from Lemma~\ref{Maximum principle}.
\end{proof}

{}From Theorem~\ref{Theorem4.1GHZ}, the next result follows.

\begin{corollary}\label{Corollary4.2GHZ}
Let $f(x,s)$ satisfy
\begin{equation*}
\lim_{s\to0^{+}}\dfrac{f(x,s)}{s}=0, \quad \text{uniformly a.e. } x\in I.
\end{equation*}
Assume $(H)$ and suppose that, for each $i\in\{1,\ldots,n\}$, there exists a compact interval $J_{i}\subseteq I_{i}$
such that
\begin{equation*}
\lim_{s\to+\infty}\dfrac{f(x,s)}{s}=+\infty, \quad \text{uniformly a.e. } x\in J_{i}.
\end{equation*}
Then there exists at least a positive solution of the two-point BVP \eqref{two-pointBVP}.
\end{corollary}
\begin{proof}
The assumption $f(x,s)/s\to 0$ as $s\to 0^{+}$ clearly implies $(f^{*})$, $(f_{0}^{-})$ and $(f_{0}^{+})$ with $q_{0}(x)\equiv K_{0}$,
where $0 < K_{0} < (\pi/L)^{2}$. Moreover, setting
\begin{equation*}
q_{\infty}^{i}(x) =
\begin{cases}
\, K_{\infty}, & \mbox{for }  x\in J_{i}; \\
\, 0,          & \mbox{for }  x\in I_{i}\setminus J_{i};
\end{cases}
\end{equation*}
with $K_{\infty} > \max_{i=1,\ldots,n} (\pi/{\it meas}(J_{i}))^{2}$,
we have $(f_{\infty})$ satisfied as well. The conclusion follows from Theorem~\ref{Theorem4.1GHZ}.
\end{proof}

The hypothesis $(f_{\infty})$ requires to control from below the growth of $f(x,s)/s$ at infinity, on \textit{each} of the
intervals $I_{i}$. In this context, a natural question which can be raised is whether a condition like $(f_{\infty})$
can be assumed only on \textit{one} of the intervals. As a partial answer we provide a result where
we consider a weaker condition in place of hypothesis $(f_{\infty})$,
namely we assume the condition only on a closed subinterval $J\subseteq I$, as in Lemma~\ref{lemmaRJ}.
In order to achieve an existence result, we add a supplementary condition of almost linear growth of $f(x,s)$ in $I\setminus J$.

\begin{theorem}\label{Th+LinCond}
Let $f(x,s)$ satisfy $(f^{*})$, $(f_{0}^{-})$, $(f_{0}^{+})$.
Let $J \subseteq I$ be a closed subinterval such that
\begin{equation*}
f(x,s)\geq 0, \quad \text{a.e. } x\in J, \; \forall \, s\geq0.
\end{equation*}
Assume the following conditions:
\begin{itemize}
\item [$(f_{\infty}^{J})$]
there exists a measurable function $q_{\infty}\in L^{1}(J,{\mathbb{R}}^{+})$ with $q_{\infty}\not\equiv0$, such that
\begin{equation*}
\liminf_{s\to+\infty}\dfrac{f(x,s)}{s}\geq q_{\infty}(x), \quad \text{uniformly a.e. } x\in J,
\end{equation*}
and
\begin{equation*}
\mu_{1}^{J}(q_{\infty})<1;
\end{equation*}
\item [$(G)$]
there exist $a,b\in L^{1}(I\setminus J,{\mathbb{R}}^{+})$ and $C>0$ such that
\begin{equation*}
|f(x,s)|\leq a(x)+b(x)s, \quad \text{a.e. } x\in I\setminus J, \; \forall \, s\geq C.
\end{equation*}
\end{itemize}
Then there exists at least a positive solution of the two-point BVP \eqref{two-pointBVP}.
\end{theorem}

\begin{proof}
As in Lemma~\ref{lemmaRJ}, set $J:=\mathopen[x_{1},x_{2}\mathclose]$. We define the set
\begin{equation*}
\Omega_{J}:=\{u\in\mathcal{C}(I)\colon |u(x)|< R_{J}, \, \text{for all } x\in J\},
\end{equation*}
where $R_{J}>0$ is as in Lemma~\ref{lemmaRJ}. Note that $\Omega_{J}$ is open and not bounded
(unless we are in the trivial case $J = I$).

Define $\lambda_{J}:=\mu_{1}^{J}(1) = (\pi/{\it meas}(J))^{2}$. Along the proof, we denote by $\varphi$ the positive eigenfunction of
\begin{equation*}
\begin{cases}
\, \varphi''+\lambda_{J}\varphi=0 \\
\, \varphi(x_{1})=\varphi(x_{2})=0,
\end{cases}
\end{equation*}
with $\|\varphi\|_{\infty}=1$. Then $\varphi(x) > 0$, for all $x\in\mathopen]x_{1},x_{2}\mathclose[$, and $\varphi'(x_{2})<0<\varphi'(x_{1})$.

We denote by $\mathbbm{1}_{J}$ the characteristic function of the interval $J$. We set $v(x):=\int_{I}G(x,s)\varphi(s)\mathbbm{1}_{J}(s) ~\!ds$
and we define $F\colon \mathcal{C}(I)\times\mathopen[0,+\infty\mathclose[\to \mathcal{C}(I)$, as
\begin{equation*}
F(u,\alpha):=\Phi(u)+\alpha v.
\end{equation*}
To reach the conclusion as in Theorem~\ref{Theorem4.1GHZ}, we have to prove that
the triplet $(Id-\Phi,\Omega_{J},0)$ is admissible and
\begin{equation*}
deg(Id-\Phi,\Omega_{J},0)=0.
\end{equation*}
To this end we show
that conditions $(i)$, $(ii)$, $(iii)$ of Theorem~\ref{deFigueiredo2} are satisfied.
It is obvious that $F(u,0)=\Phi(u)$, for all $u\in\mathcal{C}(I)$. Hence $(i)$ is valid.

Preliminary to the proof of $(ii)$ and $(iii)$, we observe that
any nontrivial solution $u\in \mathcal{C}(I)$ of the operator equation
\begin{equation*}
u = \Phi(u) + \alpha v,\quad \mbox{for } \alpha \geq 0,
\end{equation*}
is a solution of
$u'' + \tilde{f}(x,u) + \alpha \varphi(x)\mathbbm{1}_{J}(x) = 0$ with $u(0) = u(L) = 0$.
By the first part of Lemma~\ref{Maximum principle}
we know that $u(x)\geq 0$ for all $x\in I$. Hence $u$ is a non-negative solution of
\eqref{BVP2} with
\begin{equation*}
g(x,s) = f(x,s) + \alpha \varphi(x)\mathbbm{1}_{J}(x).
\end{equation*}
By definition, we have that $g(x,s) \geq f(x,s)$ for a.e.~$x\in J$ and for all $s\geq 0$,
and also $g(x,s) = f(x,s)$ for a.e.~$x\in I\setminus J$ and for all $s\geq 0$.

\medskip

\noindent
\textit{Proof of $(ii)$.\,}
Fix $\alpha\geq0$. Suppose that there exist $u\in\overline{\Omega_{J}}$ and $0\leq\zeta\leq\alpha$
satisfying $u=F(u,\zeta)$. Clearly $u\in\Omega_{J}$, by the choice of $R_{J}$
and by Lemma~\ref{lemmaRJ}. We first prove that  $|u'(x)|$ is bounded on $J$.
Using the fact that
\begin{equation*}
|u''(x)|=|f(x,s)+\zeta\varphi(x)\mathbbm{1}_{J}(x)|\leq\gamma_{R_{J}}(x)+\alpha, \quad \text{a.e. } x\in J,
\end{equation*}
we obtain that for all $y_{1},y_{2}\in J$
\begin{equation*}
|u'(y_{1})-u'(y_{2})|\leq\biggl{|}\int_{y_{1}}^{y_{2}}(\gamma_{R_{J}}(\xi)+\alpha)~\!d\xi\biggr{|}
\leq\|\gamma_{R_{J}}\|_{L^{1}(J)}+\alpha(x_{2}-x_{1}).
\end{equation*}

Now we show that there exists $\hat{x}\in J$ such that $|u'(\hat{x})|\leq {R_{J}}/{(x_{2}-x_{1})}$.
By contradiction, suppose that
\begin{equation*}
|u'(x)|>\dfrac{R_{J}}{x_{2}-x_{1}}, \quad \forall \, x\in J.
\end{equation*}
Without loss of generality, suppose $u'>0$ on $J$ (the opposite case is analogous). Then
\begin{equation*}
R_{J}=\dfrac{R_{J}}{x_{2}-x_{1}}(x_{2}-x_{1})<\int_{x_{1}}^{x_{2}}u'(\xi)~\!d\xi=u(x_{2})-u(x_{1})\leq u(x_{2})\leq R_{J},
\end{equation*}
a contradiction.

Then, for all $x\in J$, we have
\begin{equation*}
|u'(x)|\leq |u'(\hat{x})|+|u'(x)-u'(\hat{x})| \leq \dfrac{R_{J}}{x_{2}-x_{1}}+\|\gamma_{R_{J}}\|_{L^{1}(I)}+\alpha(x_{2}-x_{1})=:K,
\end{equation*}
where $K$ is a constant depending on $J$, $R_{J}$ and $\alpha$. As a consequence,
\begin{equation*}
\|(u(x),u'(x))\| \leq K^{*}:= R_{J} + K, \quad \forall \, x\in J,
\end{equation*}
where we use $\|(\xi_{1},\xi_{2})\| = |\xi_{1}| + |\xi_{2}|$ as a standard norm in ${\mathbb{R}}^{2}$.

Now, recalling the Carath\'{e}odory condition on $|f(x,s)|$, we rewrite hypothesis $(G)$ in this form:
\textit{there exist $a_{1},b_{1}\in L^{1}(\mathopen[0,x_{1}\mathclose],{\mathbb{R}}^{+})$,
$a_{2},b_{2}\in L^{1}(\mathopen[x_{2},L\mathclose],{\mathbb{R}}^{+})$, such that, for all $s\geq 0$,
\begin{equation*}
\begin{aligned}
   & |f(x,s)|\leq a_{1}(x)+b_{1}(x)s, \quad \text{a.e. } x\in \mathopen[0,x_{1}\mathclose];
\\ & |f(x,s)|\leq a_{2}(x)+b_{2}(x)s, \quad \text{a.e. } x\in \mathopen[x_{2},L\mathclose].
\end{aligned}
\end{equation*}
}
Suppose $x_{2}<L$. For all $x\in\mathopen]x_{2},L\mathclose]$ we have $\mathbbm{1}_{J}(x)=0$ and then
\begin{equation*}
\begin{aligned}
& \|(u(x),u'(x))\| = |u(x)| + |u'(x)| = u(x) + |u'(x)| =
\\ &  = u(x_{2})+\int_{x_{2}}^{x}u'(\xi)~\!d\xi
       +\biggl|u'(x_{2})+\int_{x_{2}}^{x}\bigl{[}-f(\xi,u(\xi))-\zeta\varphi(\xi)\mathbbm{1}_{J}(\xi)\bigr{]}~\!d\xi\biggr{|}
\\ & \leq  \|(u(x_{2}),u'(x_{2}))\| + \int_{x_{2}}^{x}|u'(\xi)|~\!d\xi+\int_{x_{2}}^{x}|f(\xi,u(\xi))|~\!d\xi
\\ & \leq  K^{*} + \int_{x_{2}}^{x}|u'(\xi)|~\!d\xi + \int_{x_{2}}^{x}\bigr[a_{2}(\xi)+b_{2}(\xi)u(\xi)\bigr]~\!d\xi
\\ & \leq  K^{*} + \|a_{2}\|_{L^{1}(\mathopen[x_{2},L\mathclose])} +
\int_{x_{2}}^{x}(b_{2}(\xi)+1)\|(u(\xi),u'(\xi))\|~\!d\xi.
\end{aligned}
\end{equation*}
Define
\begin{equation*}
R_{\alpha}^{2}:=(K^{*}  +\|a_{2}\|_{L^{1}(\mathopen[x_{2},L\mathclose])})\,
e^{\|1+b_{2}\|_{L^{1}(\mathopen[x_{2},L\mathclose])}}
\end{equation*}
(observe that $K^{*}$ depends on $\alpha$).
By Gronwall's inequality, we have
\begin{equation*}
0\leq u(x)\leq \|(u(x),u'(x))\| \leq R_{\alpha}^{2}, \quad \forall \, x\in\mathopen[x_{2},L\mathclose].
\end{equation*}
If $x_{1}>0$, we achieve a similar upper bound (denoted by $R_{\alpha}^{1}$) for $u(x)$ on $\mathopen[0,x_{1}\mathclose]$. The proof is analogous
and therefore is omitted.
We conclude that $F$ satisfies condition $(ii)$ of Theorem~\ref{deFigueiredo2}, with $R_{\alpha}:=\max\{R_{\alpha}^{1},R_{\alpha}^{2}\}$.

\medskip

\noindent
\textit{Proof of $(iii)$.\,} Let us fix a constant $\alpha_{0}$ with
\begin{equation*}
\alpha_{0} > \dfrac{\lambda_{J}R_{J}(x_{2}-x_{1}) +\|\gamma_{R_{J}}\|_{L^{1}(J)}}{\|\varphi\|^{2}_{L^{2}(J)}}.
\end{equation*}
Suppose by contradiction that there exist $\tilde{u}\in\overline{\Omega_{J}}$ and $\tilde{\alpha}\geq\alpha_{0}$
such that $\tilde{u}=F(\tilde{u},\tilde{\alpha})$. Then, we obtain
\begin{equation*}
\begin{aligned}
0  & \geq \tilde{u}(x_{2})\varphi'(x_{2})-\tilde{u}(x_{1})\varphi'(x_{1})
     = \Bigl{[}\tilde{u}(x)\varphi'(x)-\tilde{u}'(x)\varphi(x)\Bigr{]}_{x=x_{1}}^{x=x_{2}}
\\ & = \int_{x_{1}}^{x_{2}}\dfrac{d}{dx}\Bigl{[}\tilde{u}(x)\varphi'(x)-\tilde{u}'(x)\varphi(x)\Bigr{]} ~\!dx
     = \int_{J}\Bigl{[}\tilde{u}(x)\varphi''(x)-\tilde{u}''(x)\varphi(x)\Bigr{]} ~\!dx
\\ & = \int_{J}\Bigl{[}-\lambda_{J}\tilde{u}(x)\varphi(x)+f(x,\tilde{u}(x))\varphi(x)+\tilde{\alpha}(\varphi(x))^{2}\Bigr{]} ~\!dx
\\ & \geq -\lambda_{J}R_{J}(x_{2}-x_{1})-\|\gamma_{R_{J}}\|_{L^{1}(J)} + \tilde{\alpha}\|\varphi\|^{2}_{L^{2}(J)}
> 0,
\end{aligned}
\end{equation*}
a contradiction. Hence $F$ satisfies $(iii)$.

We have thus verified all the conditions in Theorem~\ref{deFigueiredo2}. This concludes the proof.
\end{proof}

Up to now, we have essentially reconsidered in an explicit topological degree setting the existence results obtained
in \cite{GaudenziHabetsZanolin3} by means of lower and upper solutions techniques.
Our next goal is to produce multiplicity results by taking advantage of the previous lemmas
and exploiting the excision property of the Leray-Schauder degree
in order to provide more precise information about the localization of the solutions.

\section{Multiplicity results}\label{section4}

In this section we propose an approach, based on the additivity property of the Leray-Schauder degree, in order to provide
sharp multiplicity results for positive solutions of
\begin{equation}\label{two-pointBVP4}
\begin{cases}
\, u''+f(x,u)=0 \\
\, u(0)=u(L)=0.
\end{cases}
\end{equation}
Throughout the section we suppose that $f(x,s)$ is a Carath\'{e}odory function satisfying $(f^{*})$, $(f_{0}^{-})$, $(f_{0}^{+})$,
as well as $(H)$ with $(f_{\infty})$. Recall also that, in view of the discussion in Section~\ref{section2}, the positive solutions
of \eqref{two-pointBVP4} are the nontrivial fixed points of the completely continuous operator $\Phi\colon\mathcal{C}(I)\to\mathcal{C}(I)$
defined in \eqref{operator}.

We introduce now some notation. Let $\mathcal{I}\subseteq\{1,\ldots,n\}$ be a subset of indices (possibly empty) and let
$r,R$ be two fixed constants with
\begin{equation*}
0 < r \leq r_{0} < R^{*} \leq R,
\end{equation*}
where $r_{0}$ and $R^{*}$ are the constants obtained in Lemma~\ref{lemmar0} and Lemma~\ref{lemmaR}, respectively.
We define two families of open and possibly unbounded sets
\begin{equation*}
\begin{aligned}
\Omega^{\mathcal{I}}:=
\biggl\{\,u\in\mathcal{C}(I)\colon
   & \max_{x\in I_{i}}|u(x)|<R, \, i\in\mathcal{I};                             &
\\ & \max_{x\in I_{i}}|u(x)|<r, \, i\in\{1,\ldots,n\}\setminus\mathcal{I}       & \biggr\}
\end{aligned}
\end{equation*}
and
\begin{equation*}
\begin{aligned}
\Lambda^{\mathcal{I}}:=
\biggl\{\,u\in\mathcal{C}(I)\colon
   & r<\max_{x\in I_{i}}|u(x)|<R, \, i\in\mathcal{I};                           &
\\ &   \max_{x\in I_{i}}|u(x)|<r, \, i\in\{1,\ldots,n\}\setminus\mathcal{I}     & \biggr\}.
\end{aligned}
\end{equation*}
We note that, for each $\mathcal{I}\subseteq \{1,\ldots,n\}$, we have
\begin{equation*}
\Omega^{\mathcal{I}}=\bigcup_{\mathcal{J}\subseteq\mathcal{I}}\Lambda^{\mathcal{J}},
\end{equation*}
and the union is disjoint, since $\Lambda^{\mathcal{J}'} \cap \Lambda^{\mathcal{J}''} = \emptyset$, for $\mathcal{J}'\neq\mathcal{J}''$.

\noindent
We also observe that for ${\mathcal{I}} = \emptyset$, we have
\begin{equation*}
\Lambda^{\emptyset} = \Omega^{\emptyset} =
\biggl\{u\in\mathcal{C}(I)\colon
\max_{x\in I_{i}}|u(x)|<r, \, i\in\{1,\ldots,n\} \biggr\} \supseteq B(0,r).
\end{equation*}
By the maximum principle (Lemma~\ref{Maximum principle}) any solution $u\in {\text{cl}}(\Lambda^{\emptyset})$
of the operator equation $u = \Phi(u)$ is a (non-negative) solution
of \eqref{two-pointBVP4} such that $0\leq u(x) \leq r$, for all $x\in\bigcup_{i =1}^{n} I_{i}$.
On the other hand, we know that $u(x)$ is convex in each interval contained in $I\setminus\bigcup_{i =1}^{n} I_{i}$ and thus we conclude that
$0 \leq u(x) \leq r$, $\forall \, x\in I$, so that $u\in B[0,r]$. Lemma~\ref{lemmar0}, Lemma~\ref{lemmar}
and the choice of $r\in\mathopen]0,r_{0}\mathclose]$ then imply
\begin{equation}\label{deg1}
deg(Id-\Phi,\Lambda^{\emptyset},0)=deg(Id-\Phi,B(0,r),0)=1.
\end{equation}
The above relation shows that even if $\Lambda^{\mathcal{I}}$ is an unbounded open set, then, at least for ${\mathcal{I}} = \emptyset$,
the topological degree is well defined. The next result is the key lemma to provide the existence
of nontrivial fixed points (and hence multiplicity results) whenever the topological degree is defined on the sets $\Lambda^{\mathcal{I}}$ and
$\Omega^{\mathcal{I}}$.

\begin{lemma}\label{lemma-1}
Let $\mathcal{I}\subseteq\{1,\ldots,n\}$ be a set of indices.
Suppose that for all ${\mathcal{J}}\subseteq {\mathcal{I}}$,
the triplets $(Id-\Phi,\Lambda^{\mathcal{J}},0)$ and
$(Id-\Phi,\Omega^{\mathcal{J}},0)$ are admissible  with
\begin{equation}\label{deg2}
deg(Id-\Phi,\Omega^{\mathcal{J}},0)=0, \quad \forall \, \emptyset\neq\mathcal{J}\subseteq {\mathcal{I}}.
\end{equation}
Then
\begin{equation}\label{deg3}
deg(Id-\Phi,\Lambda^{\mathcal{I}},0)=(-1)^{\#\mathcal{I}}.
\end{equation}
\end{lemma}

\begin{proof}
First of all, we notice that, in view of \eqref{deg1}, the conclusion is trivially satisfied when ${\mathcal{I}} = \emptyset$.
Suppose now that $m:=\#{\mathcal{I}} \geq 1$.
We are going to prove our claim by using an inductive argument.
More precisely, for every integer $k$ with $0\leq k\leq m$, we introduce the property ${\mathscr{P}(k)}$
which reads as follows:
\textit{the formula
\begin{equation*}
deg(Id-\Phi,\Lambda^{\mathcal{J}},0)=(-1)^{\#\mathcal{J}}
\end{equation*}
holds for each subset ${\mathcal{J}}$ of ${\mathcal{I}}$ having at most $k$ elements.}
In this manner, if we are able to prove ${\mathscr{P}}(m)$, then \eqref{deg3} follows.

\noindent
\textit{Verification of ${\mathscr{P}}(0)$.} See \eqref{deg1}.

\noindent
\textit{Verification of ${\mathscr{P}}(1)$.}
For ${\mathcal{J}} =\emptyset$ the result is already proved in \eqref{deg1}.
For ${\mathcal{J}}=\{j\}$, with $j\in {\mathcal{I}}$, we have
\begin{equation*}
\begin{aligned}
deg(Id-\Phi,\Lambda^{\mathcal{J}},0)
    & = deg(Id-\Phi,\Lambda^{\{j\}},0)=deg(Id-\Phi,\Omega^{\{j\}}\setminus \Lambda^{\emptyset},0)
\\  & = 0-1=-1  =(-1)^{\#\mathcal{J}}.
\end{aligned}
\end{equation*}

\noindent
\textit{Verification of ${\mathscr{P}}(k-1)\Rightarrow{\mathscr{P}}(k)$, for $1\leq k\leq m$.}
Assuming the validity of ${\mathscr{P}}(k-1)$ we have that the formula is true for every subset of ${\mathcal{I}}$
having at most $k-1$ elements. Therefore, in order to prove ${\mathscr{P}}(k)$, we have only to check that the formula is true for an
arbitrary subset $\mathcal{J}$ of $\mathcal{I}$ with $\#{\mathcal{J}} = k$.
Writing $\Omega^{\mathcal{J}}$ as the disjoint union
\begin{equation*}
\Omega^{\mathcal{J}} = \Lambda^{\mathcal{J}} \cup \bigcup_{\mathcal{K}\subsetneq\mathcal{J}}\Lambda^{\mathcal{K}},
\end{equation*}
by the inductive hypothesis, we obtain
\begin{equation*}
\begin{aligned}
&       deg(Id-\Phi,\Lambda^{\mathcal{J}},0)
      = deg(Id-\Phi,\Omega^{\mathcal{J}},0)-\sum_{\mathcal{K}\subsetneq\mathcal{J}}deg(Id-\Phi,\Lambda^{\mathcal{K}},0) =
\\  & = 0-\sum_{\mathcal{K}\subsetneq\mathcal{J}}(-1)^{\#\mathcal{K}}
      = -\sum_{\mathcal{K}\subseteq\mathcal{J}}(-1)^{\#\mathcal{K}}+(-1)^{\#\mathcal{J}}.
\end{aligned}
\end{equation*}
Observe now that
\begin{equation*}
\sum_{\mathcal{K}\subseteq\mathcal{J}}(-1)^{\#\mathcal{K}} = 0,
\end{equation*}
due to the fact that in a finite set there are so many subsets of even cardinality how many subsets of odd cardinality.
Thus we conclude that
\begin{equation*}
deg(Id-\Phi,\Lambda^{\mathcal{J}},0) = (-1)^{\#\mathcal{J}}
\end{equation*}
and ${\mathscr{P}}(k)$ is proved.
\end{proof}

In order to apply Lemma~\ref{lemma-1} we have to check assumption \eqref{deg2}. The next result provides sufficient conditions for
\eqref{deg2}. To this aim, we introduce a third family of unbounded sets, defined as follows
\begin{equation*}
\Gamma^{\mathcal{I}}:=
\Bigl\{u\in\mathcal{C}(I)\colon \max_{x\in I_{i}}|u(x)|<r, \, i\in\{1,\ldots,n\}\setminus\mathcal{I} \Bigr\},
\end{equation*}
where $\mathcal{I}\subseteq\{1,\ldots,n\}$. It is also convenient to consider Carath\'{e}odory functions
$g\colon I\times{\mathbb{R}}^{+}\to{\mathbb{R}}$ such that
\begin{equation}\label{cond_g}
\begin{aligned}
   &  g(x,s) \geq f(x,s), \quad \text{a.e. } x\in \bigcup_{i\in {\mathcal{I}}} I_{i}, \; \forall \, s\geq 0;
\\ &  g(x,s) = f(x,s),    \quad \text{a.e. } x\in I \setminus \bigcup_{i\in {\mathcal{I}}} I_{i}, \; \forall \, s\geq 0.
\end{aligned}
\end{equation}
Using a standard procedure, for any given $g(x,s)$ as above, we
define a completely continuous operator $\Psi_{g}\colon\mathcal{C}(I)\to\mathcal{C}(I)$ as
\begin{equation*}
(\Psi_{g} u)(x):= \int_{I} G(x,\xi)\tilde{g}(\xi,u(\xi))~\!d\xi,
\end{equation*}
where
\begin{equation*}
\tilde{g}(x,s)=
\begin{cases}
\, g(x,s), & \text{if } s\geq0; \\
\, g(x,0),  & \text{if } s\leq0.
\end{cases}
\end{equation*}
Notice that $g(x,0) = 0$, for a.e.~$x\notin\bigcup_{i\in {\mathcal{I}}} I_{i}$, while $g(x,0) \geq 0$, for a.e.
$x\in\bigcup_{i\in {\mathcal{I}}} I_{i}$. In this manner, the first part of Lemma~\ref{Maximum principle} applies for
$h(x,s) := \tilde{g}(x,s)$.

\begin{lemma}\label{lemma0}
Let $\mathcal{I}\subseteq\{1,\ldots,n\}$, with $\mathcal{I}\neq\emptyset$.
Suppose that the triplet $(Id-\Psi_{g},\Gamma^{\mathcal{I}},0)$ is admissible
for every Carath\'{e}odory function $g\colon I\times{\mathbb{R}}^{+}\to{\mathbb{R}}$ satisfying \eqref{cond_g}.
Then
\begin{equation*}
deg(Id-\Phi,\Omega^{\mathcal{I}},0)=0.
\end{equation*}
\end{lemma}

\begin{proof}
The proof combines some arguments previously developed along the proofs of Lemma~\ref{lemmaR} and Theorem~\ref{Th+LinCond}.
In order to simplify the notation we set $A:=\bigcup_{i\in\mathcal{I}} I_{i}$.

For each index $i\in\mathcal{I}$, we define $\lambda_{I_{i}}:= \mu_{1}^{I_{i}}(1) = (\pi/{\it meas}(I_{i}))^{2}$ and
we denote by $\varphi_{i}$ the positive eigenfunction of
\begin{equation*}
\begin{cases}
\, \varphi''+\lambda_{I_{i}}\varphi=0 \\
\, \varphi|_{\partial I_{i}}=0,
\end{cases}
\end{equation*}
with $\|\varphi\|_{\infty}=1$.
Then $\varphi_{i}(x)\geq 0$, for all $x\in I_{i}$.
We denote by $\mathbbm{1}_{I_{i}}$ the characteristic function of the interval $I_{i}$.

Next, we define
\begin{equation*}
v(x):= \int_{I}G(x,s)\biggl{(}\sum_{i\in\mathcal{I}}\varphi_{i}(s)\mathbbm{1}_{I_{i}}(s)\biggr{)} ~\!ds
\end{equation*}
and introduce the operator
$F\colon \mathcal{C}(I)\times\mathopen[0,+\infty\mathclose[\to \mathcal{C}(I)$, as
\begin{equation*}
F(u,\alpha):=\Phi(u)+\alpha v.
\end{equation*}

To prove our claim, we check $(ii)$ and $(iii)$ of Theorem~\ref{deFigueiredo2}
(clearly, $F(u,0) = \Phi(u)$, so that $(i)$ is trivially satisfied).

\medskip

\noindent
\textit{Proof of $(ii)$.\,} Fix $\alpha \geq 0$.
By the definition of $v(x)$ and the first part of Lemma~\ref{Maximum principle}, any nontrivial solution $u$ of
\begin{equation*}
u = F(u,\zeta),\quad \mbox{with } u\in {\text{cl}}(\Omega^{\mathcal{I}}), \; \zeta\in\mathopen[0,\alpha\mathclose],
\end{equation*}
is a non-negative solution of
\begin{equation}\label{BVP4.6}
\begin{cases}
\, u''+g(x,u)=0 \\
\, u(0)=u(L)=0
\end{cases}
\end{equation}
with
\begin{equation*}
g(x,s):= f(x,s) + \zeta \sum_{i\in\mathcal{I}}\varphi_{i}(x)\mathbbm{1}_{I_{i}}(x).
\end{equation*}
The hypothesis $u\in{\text{cl}}(\Omega^{\mathcal{I}})$ implies that $0 \leq u(x) \leq r$ for all $x\in I_{i}$,
if $i\notin\mathcal{I}$, and $0 \leq u(x) \leq R$ for all $x\in I_{i}$, if  $i\in\mathcal{I}$.

We first note that $0\leq u(x)<R$, for every $x\in A$, by the choice of $R \geq R^{*}$. Moreover the
admissibility of the triplets $(Id-\Psi_{g},\Gamma^{\mathcal{I}},0)$ implies that any $\Psi_{g}$ has no fixed points on $\partial\Gamma^{\mathcal{I}}$.
Then each non-negative solution of \eqref{BVP4.6} satisfies $0\leq u(x)<r$, for all $x\in\bigcup_{i\notin\mathcal{I}}I_{i}$.
We deduce that $u\in\Omega^{\mathcal{I}}$.

Since $g(x,s) = f(x,s) \leq 0$ for a.e.~$x\in I\setminus \bigcup_{i=1}^{n}I_{i}$ and for all $s\geq 0$,
by convexity, we conclude that
\begin{equation*}
\|u\|_{\infty} = \max_{x\in I} u(x) \leq R.
\end{equation*}
Then $(ii)$ is proved with $R_{\alpha} := R$.

\medskip

\noindent
\textit{Proof of $(iii)$.\,} Let us fix a constant $\alpha_{0}$ with
\begin{equation*}
\alpha_{0} > \max_{i\in\mathcal{I}} \dfrac{\lambda_{I_{i}}R L +\|\gamma_{R}\|_{L^{1}(I)}}{\|\varphi_{i}\|^{2}_{L^{2}(I_{i})}}.
\end{equation*}
Suppose, by contradiction, that there exist $\tilde{u}\in{\text{cl}}(\Omega^{\mathcal{I}})$ and $\tilde{\alpha} \geq \alpha_{0}$
such that $\tilde{u} = F(\tilde{u},\tilde{\alpha}) =\Phi(\tilde{u}) + \tilde{\alpha} v$.
Since $\Phi(0) = 0$, we have $\tilde{u}\not\equiv0$ and, as in the previous step, $\tilde{u}(x)$
is a non-negative solution of \eqref{BVP4.6} for
\begin{equation*}
g(x,s):= f(x,s) + \tilde{\alpha} \sum_{i\in\mathcal{I}}\varphi_{i}(x)\mathbbm{1}_{I_{i}}(x).
\end{equation*}

By assumption, $\mathcal{I}\neq\emptyset$. So, let us fix an index $k\in\mathcal{I}$ and set
$I_{k}:=\mathopen[x_{1},x_{2}\mathclose]$. Arguing as in the proof of Theorem~\ref{Th+LinCond}, we obtain
\begin{equation*}
\begin{aligned}
0  & \geq \tilde{u}(x_{2})\varphi'_{k}(x_{2})-\tilde{u}(x_{1})\varphi'_{k}(x_{1})
     = \Bigl{[}\tilde{u}(x)\varphi'_{k}(x)-\tilde{u}'(x)\varphi_{k}(x)\Bigr{]}_{x=x_{1}}^{x=x_{2}}
\\ & = \int_{I_{k}}\dfrac{d}{dx}\Bigl{[}\tilde{u}(x)\varphi'_{k}(x)-\tilde{u}'(x)\varphi_{k}(x)\Bigr{]} ~\!dx
     = \int_{I_{k}}\Bigl{[}\tilde{u}(x)\varphi''_{k}(x)-\tilde{u}''(x)\varphi_{k}(x)\Bigr{]} ~\!dx
\\ & = \int_{I_{k}}\Bigl{[}-\lambda_{I_{k}}\tilde{u}(x)\varphi_{k}(x)+f(x,\tilde{u}(x))\varphi_{k}(x) + \tilde{\alpha}(\varphi_{k}(x))^{2}\Bigr{]} ~\!dx
\\ & \geq -\lambda_{I_{k}}R(x_{2}-x_{1})-\|\gamma_{R}\|_{L^{1}(I_{k})} + \tilde{\alpha}\|\varphi_{k}\|^{2}_{L^{2}(I_{k})}
> 0,
\end{aligned}
\end{equation*}
a contradiction. Hence $F$ satisfies $(iii)$.
\end{proof}

Putting together Lemma~\ref{lemma-1} and Lemma~\ref{lemma0}
we can obtain results of multiplicity of positive solutions provided that we are able to show
that the topological degree on certain open sets is well defined. With this respect, observe that from
Lemma~\ref{lemmaRJ} we know that there are no positive solutions $u(x)$ with $\max_{x\in I_{i}} u(x) \geq R$. Thus,
we only have to show that the level $r$ is not achieved by the solutions $u(x)$ of $u = \Psi_{g}(u)$
for $x$ in some of the intervals $I_{i}$.

\begin{theorem}\label{Th-Multiplicity}
Let $f\colon\mathopen[0,L\mathclose]\times{\mathbb{R}}^{+}\to {\mathbb{R}}$ be an $L^{1}$-Carath\'{e}odory function
satisfying $(f^{*})$, $(f_{0}^{-})$, $(f_{0}^{+})$ and $(H)$ with $(f_{\infty})$.
Suppose that for every Carath\'{e}odory function $g\colon I\times{\mathbb{R}}^{+}\to{\mathbb{R}}$ satisfying \eqref{cond_g} and
for every $\emptyset\neq\mathcal{I}\subseteq\{1,\ldots,n\}$
the triplet $(Id-\Psi_{g},\Gamma^{\mathcal{I}},0)$ is admissible.
Then there exist at least $2^{n}-1$ positive solutions of the two-point BVP \eqref{two-pointBVP}.
\end{theorem}

\begin{proof}
First of all, we claim that the triplet $(Id-\Phi,\Lambda^{\mathcal{I}},0)$ is admissible
for all $\mathcal{I}\subseteq\{1,\ldots,n\}$. Indeed,
if $\mathcal{I}=\emptyset$ this is clear from \eqref{deg1}.
If $\mathcal{I}\neq\emptyset$, the claim follows since all possible fixed points of $\Phi$ are contained in $B(0,R)$
(as already observed) and they can not achieve the radius $r$ by virtue of the admissibility of $(Id-\Psi_{g},\Gamma^{\mathcal{I}},0)$
for all $\emptyset\neq\mathcal{I}\subseteq\{1,\ldots,n\}$.

{}From Lemma~\ref{lemma-1} and Lemma~\ref{lemma0} it follows that
\begin{equation*}
deg(Id-\Phi,\Lambda^{\mathcal{I}},0)\neq0, \quad \text{for all } \, \mathcal{I}\subseteq\{1,\ldots,n\}.
\end{equation*}
We obtain the claim noting that $0\notin\Lambda^{\mathcal{I}}$, for all $\emptyset\neq\mathcal{I}\subseteq\{1,\ldots,n\}$,
and using the fact that the number of nonempty subsets of a set with $n$ elements is $2^{n}-1$.
\end{proof}

\section{A special case}\label{section5}

In this section we provide an application of the existence and multiplicity results
obtained in Section~\ref{section3} and Section~\ref{section4} to the search of positive solutions
for a two-point BVP of the form
\begin{equation}\label{two-pointBVPag}
\begin{cases}
\, u''+a(x)g(u)=0 \\
\, u(0)=u(L)=0,
\end{cases}
\end{equation}
where we suppose that
$g\colon{\mathbb{R}}^{+}\to{\mathbb{R}}^{+}$ is a continuous function such
that
\begin{equation}\label{eq-g}
g(0)=0 \quad \mbox{and} \quad g(s)>0, \; \forall \, s>0.
\end{equation}
With the aim of providing a simplified exposition of our main result, we suppose that the weight function
$a\colon I\to{\mathbb{R}}$ is \textit{continuous}.
The more general case of an $L^{1}$-weight function can be treated as well with
minor modifications in the statements of the theorems (this will be briefly discussed in a
final remark). Since we are looking for \textit{positive solutions} of
\eqref{two-pointBVPag}, in order to avoid trivial situations, we suppose that
\begin{equation}\label{eq-a+}
a^{+}(x):= \max\{a(x),0\} \not\equiv 0.
\end{equation}
In the context of continuous functions  this is just the same as to assume $a(x_{0}) > 0$ for some $x_{0} \in I$.
As usual, we also set
$a^{-}(x):= \max\{-a(x),0\}$, so that $a(x) = a^{+}(x) - a^{-}(x)$.

In order to enter in the general setting of the previous sections for
\begin{equation}\label{eq-f=ag}
f(x,s):= a(x) g(s),
\end{equation}
we suppose that
\begin{equation*}
g_{0}:=\limsup_{s\to 0^{+}} \dfrac{g(s)}{s} < +\infty
\quad \mbox{and} \quad
g_{\infty}:=\liminf_{s\to +\infty} \dfrac{g(s)}{s} > 0.
\end{equation*}
In such a situation, we can extend $g(s)$ to the negative real line, by setting $g(s)=0$, $\forall \, s\leq 0$.
Then Lemma~\ref{Maximum principle} ensures that any nontrivial solution $u(x)$ of \eqref{two-pointBVPag}
satisfies $u(x) > 0$ for all $x\in \mathopen{]}0,L\mathclose{[}$ with $u'(0) > 0 > u'(L)$. Notice also that
for $f(x,s)$ defined as in \eqref{eq-f=ag} it follows that $(f_{0}^{-})$ is satisfied
for $q_{-}(x) := g_{0}\, a^{-}(x)$.

Now, we translate condition $(f_{0}^{+})$ in the new setting. From
\begin{equation*}
\limsup_{s\to 0^{+}} \dfrac{f(x,s)}{s} =
\limsup_{s\to 0^{+}} \dfrac{a(x) g(s)}{s} \leq g_{0}\, a^{+}(x), \quad \mbox{uniformly for all } x \in I,
\end{equation*}
we immediately conclude that $(f_{0}^{+})$ holds if and only if
\begin{equation*}
g_{0} < \lambda_{0},
\end{equation*}
where $\lambda_{0} > 0$ is the first eigenvalue of the eigenvalue problem
\begin{equation*}
\varphi'' + \lambda a^{+}(x) \varphi =0, \quad \varphi(0) = \varphi(L) = 0.
\end{equation*}

As a next step, we look for an equivalent formulation of conditions $(H)$ and $(f_{\infty})$ for $f(x,s)$ as in \eqref{eq-f=ag}.
Accordingly, we consider the following hypothesis on the weight function.
\begin{itemize}
\item[$(H')$] \textit{
Suppose there exists a finite sequence of $2n+2$ points in $\mathopen[0,L\mathclose]$ (possibly coincident)
\begin{equation*}
0=\tau_{0}\leq\sigma_{1}<\tau_{1}<\sigma_{2}<\tau_{2}<\ldots<\sigma_{n}<\tau_{n}\leq\sigma_{n+1}=L
\end{equation*}
such that
\begin{itemize}
\item[$\bullet\;$] $a(x)\geq0$, for all  $x\in\mathopen[\sigma_{i},\tau_{i}\mathclose]$, $i=1,\ldots,n$;
\item[$\bullet\;$] $a(x)\leq0$, for all $x\in\mathopen]\tau_{i},\sigma_{i+1}\mathclose[$, $i=0,1,\ldots,n$.
\end{itemize}
}
\end{itemize}
By assuming $(H')$ we implicitly suppose that $a(x)$ vanishes at the points
$x = \tau_{1}, \sigma_{2}, \ldots, \tau_{n-1}, \sigma_{n}$.
With a usual convention, if $\tau_{0} = \sigma_{1}$ (or $\tau_{n}= \sigma_{n+1}$)
the assumption $a(x) \leq 0$ on the first open interval (or on the last one, respectively) is vacuously satisfied.

\begin{remark}\label{rem-4.1}
The sign condition on the weight function allows the possibility that $a(x)$ may identically vanish in some subintervals
of $I$ (even infinitely many). Figure~\ref{fig-02} below shows a possible graph which is in agreement with assumption $(H')$.
$\hfill\lhd$
\end{remark}

\begin{figure}[h!]
\centering
\includegraphics[width=0.95\textwidth]{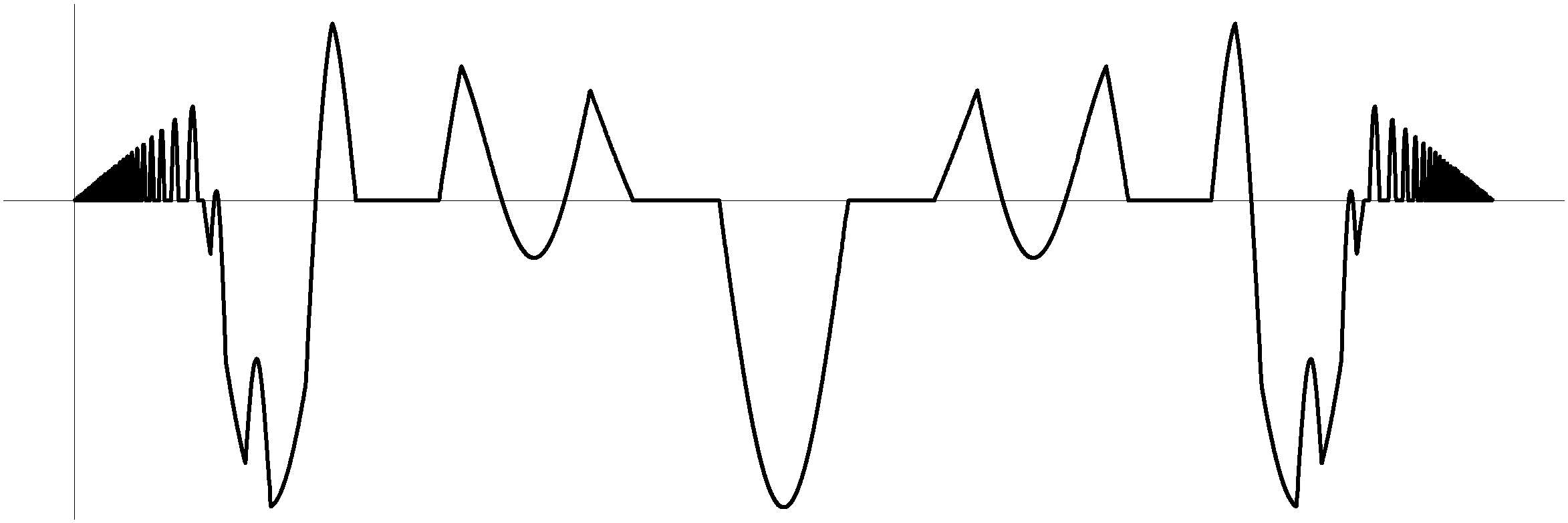}
\caption{\small{The figure shows the graph of the continuous function $a(x):=\max\{0,x(1-x)\sin(3/x(1-x))\}-\max\{0,-\sin(11x\pi)/4\}$
on $\mathopen{]}0,1\mathclose{[}$ and defined as $0$ at the endpoints. This is an example of weight function that satisfies
$(H')$ for an obvious choice of the points $\sigma_{i}$ and $\tau_{i}$
and, moreover, it has infinitely many humps.}}
\label{fig-02}
\end{figure}

Given any $a(x)$ satisfying $(H')$, consistently with the notation introduced in Section~\ref{section2}, we set
\begin{equation*}
I_{i}:=\mathopen[\sigma_{i},\tau_{i}\mathclose], \qquad i=1,\ldots,n.
\end{equation*}
For such a choice of the weight function $a(x)$, we have that $(H)$ is satisfied for $f(x,s)$ as in \eqref{eq-f=ag}.
Moreover, for every $i=1,\ldots,n$, we obtain
\begin{equation*}
\liminf_{s\to +\infty} \dfrac{f(x,s)}{s} = \liminf_{s\to +\infty} \dfrac{a(x) g(s)}{s} \geq g_{\infty} \, a(x),
\quad \mbox{uniformly for all } x \in I_{i}.
\end{equation*}
Thus we conclude that $(f_{\infty})$ holds provided that
\begin{equation*}
a(x)\not\equiv 0 \; \mbox{ on } I_{i} \quad \mbox{and} \quad g_{\infty} > \lambda_{1}^{i}, \quad \forall \, i=1,\ldots,n,
\end{equation*}
where $\lambda_{1}^{i} > 0$ is the first eigenvalue of the eigenvalue problem
\begin{equation*}
\varphi'' + \lambda a(x) \varphi =0, \quad \varphi|_{\partial I_{i}} = 0.
\end{equation*}
Notice that, as a consequence of Sturm theory (see for instance \cite{CoddingtonLevinson, Zettl}), we know that
\begin{equation*}
\lambda_{0} \leq \lambda_{1}^{i}, \quad \forall \, i=1,\ldots,n.
\end{equation*}

\subsection{Existence results}\label{subsec-5.1}

Now we are in a position to present some corollaries of the existence results in Section~\ref{section3} for problem
\eqref{two-pointBVPag}. In this context, Theorem~\ref{Theorem4.1GHZ} implies the following

\begin{theorem} \label{th-5.1}
For $g(s)$ and $a(x)$ as above, suppose that
\begin{equation*}
g_{0} < \lambda_{0}
\end{equation*}
and $a(x)\not\equiv0$ on $I_{i}$, for each $i = 1,\ldots,n$, with
\begin{equation*}
g_{\infty}  > \max_{i=1,\ldots,n} \lambda_{1}^{i}.
\end{equation*}
Then problem \eqref{two-pointBVPag} has at least one positive solution.
\end{theorem}

As an obvious corollary of Theorem~\ref{th-5.1}, we have that if $g_{0} = 0$ and $g_{\infty} = +\infty$, then a positive solution
always exists, provided that $a(x)\not\equiv0$ on $I_{i}$ (see Corollary~\ref{Corollary4.2GHZ} and also
compare to \cite[Corollary~4.2]{GaudenziHabetsZanolin3}).

\begin{remark}\label{rem-5.1}
First of all, we observe that Theorem~\ref{th-5.1} (as well as the more general Theorem~\ref{Theorem4.1GHZ}) applies
in a trivial manner if $a(x)\geq 0$ (and $a(x)\not\equiv0$) on $I$. Indeed, as already remarked
after the introduction of condition $(H)$, such hypothesis is satisfied also when $f(x,s)\geq 0$ for a.e.~$x\in I$
and for each $s \geq 0$.

In the case of a sign-changing weight function $a(x)$, namely when $a^{+}(x)\not\equiv 0$ and also
$a^{-}(x)\not\equiv 0$, the choice of the intervals $I_{i}$ is mandatory when the set $a^{-1}(0)=\{x\in I \colon a(x) = 0\}$
is made by a finite number of simple zeros. In such a situation, $a(x) > 0$ on $\mathopen]\sigma_{i},\tau_{i}\mathclose[$
and  $a(x) < 0$ on $\mathopen]\tau_{i},\sigma_{i+1}\mathclose[$. The choice of the intervals $I_{i}$ is also determined
if $a^{-1}(0)$ is finite.
However, generally speaking, there is some arbitrariness in the choice of
the way in which we separate the intervals of non-negativity to the intervals of non-positivity of $a(x)$.
This happens, for instance, when $a^{-1}(0)$ contains an interval. In such a situation, the manner in which
we define the intervals $I_{i}$ affects the computation of the eigenvalues $\lambda_{i}$ and hence the lower bound for $g_{\infty}$.

With this respect we exhibit a simple example.
Let us consider the following weight function
\begin{equation*}
a_{\varepsilon}(x) =
\begin{cases}
\, 1,  & \text{if }  x\in \Bigl{[}0,\dfrac{\pi}{2} -\varepsilon\Bigr{]}\cup \Bigl{[}\dfrac{\pi}{2} + \varepsilon, \pi\Bigr{]}; \vspace*{1pt} \\
\, 0,  & \text{if }  x\in \Bigl{]}\dfrac{\pi}{2} -\varepsilon,\dfrac{\pi}{2} + \varepsilon \Bigr{[}; \\
\, -1,     & \text{if }  x\in \mathopen]\pi,2\pi\mathclose];
\end{cases}
\end{equation*}
where $0 < \varepsilon < \pi/2$ is fixed. For convenience, we have chosen for our example a (discontinuous) step function,
however, our argument can be adapted in the continuous case via a smoothing procedure on $a_{\varepsilon}(x)$.
For this weight function we can take
$I_{1} = \mathopen{[}\sigma_{1},\tau_{1}\mathclose{]} = \mathopen{[}0,\frac{\pi}{2} -\varepsilon\mathclose{]}$
and $I_{2} = \mathopen{[}\sigma_{2},\tau_{2}\mathclose{]} = \mathopen{[}\frac{\pi}{2} + \varepsilon, \pi\mathclose{]}$.
In this situation, we have $a_{\varepsilon}(x) = 0$ on $\mathopen]\tau_{1},\sigma_{2}\mathclose[ =
\mathopen]\frac{\pi}{2} -\varepsilon,\frac{\pi}{2} + \varepsilon\mathclose[$ and $a_{\varepsilon}(x) < 0$ on
$\mathopen]\tau_{2},\sigma_{3}\mathclose] = \mathopen]\pi,2\pi\mathclose]$. Moreover,
\begin{equation*}
\lambda_{1}^{1} = \lambda_{1}^{2} = \left(\dfrac{2\pi}{\pi - 2\varepsilon}\right)^{\!2} > 4.
\end{equation*}
On the other hand, for the same weight function, we can also take
$I_{1} = \mathopen[\sigma_{1},\tau_{1}\mathclose] = \mathopen[0,\pi\mathclose]$ as unique interval of non-negativity.
To compute $\lambda_{1}^{1}$, we have to determine the first eigenvalue of
$\varphi'' + \lambda a_{\varepsilon}(x)\varphi = 0$ with $\varphi(0) = \varphi(\pi) = 0$.
For $\varepsilon > 0$ very close to zero, we find that $\lambda_{1}^{1} $ is close to $1$ (and for sure less than $4$).
As a consequence,  with this second choice of the interval, we provide a better lower bound for $g_{\infty}$.

The above example shows that Theorem~\ref{th-5.1} is a slightly more general version of \cite[Theorem~4.1]{GaudenziHabetsZanolin3},
in the sense that we can improve the lower bound on $g_{\infty}$ (at least for some particular weight functions
which vanish on their intervals of non-negativity).
$\hfill\lhd$
\end{remark}

Another way to improve the lower bound on $g_{\infty}$ of Theorem~\ref{th-5.1} is feasible by applying Theorem~\ref{Th+LinCond}.
However, this requires  to impose a further growth assumption on $g(s)$.

\begin{theorem} \label{th-5.2}
For $g(s)$ and $a(x)$ as above, suppose that
\begin{equation*}
g_{0} < \lambda_{0}
\end{equation*}
and $a(x)\not\equiv0$ on $I_{i}$, for each $i=1,\ldots,n$, with
\begin{equation*}
g_{\infty}  > \min_{i=1,\ldots,n} \lambda_{1}^{i}
\quad \mbox{and} \quad
\limsup_{s\to +\infty}\dfrac{g(s)}{s} < +\infty.
\end{equation*}
Then problem \eqref{two-pointBVPag} has at least one positive solution.
\end{theorem}

\begin{remark}\label{rem-5.2}
We have stated Theorem~\ref{th-5.2} in a form which is suitable for a comparison with Theorem~\ref{th-5.1}. Actually,
the result holds even we do not assume $(H')$, but we just suppose that there exists an interval $J\subseteq I$
where $a(x)\geq 0$ and $a(x)\not\equiv 0$, and $g_{\infty} > \lambda_{1}^{J}$, where $\lambda_{1}^{J}$ is the first eigenvalue
of the eigenvalue problem $\varphi'' + \lambda a(x) \varphi =0$, $\varphi|_{\partial J} = 0$.
$\hfill\lhd$
\end{remark}

A simple corollary of Theorem~\ref{th-5.1} can be obtained when $g_{0} = 0$ and $g_{\infty} > 0$
for the problem
\begin{equation*}
\begin{cases}
\, u''+ \nu a(x)g(u)=0 \\
\, u(0)=u(L)=0.
\end{cases}
\end{equation*}
Indeed, in such a case, we have the existence of at least one positive solution for each $\nu > 0$ sufficiently large.

\subsection{Multiplicity of positive solutions}\label{subsec-5.2}

Now we show how the main results of Section~\ref{section4} can be applied when $f(x,s) = a(x) g(s)$. To this aim,
besides \eqref{eq-a+}, we also suppose
\begin{equation}\label{eq-a-}
a^{-}(x)\not\equiv 0.
\end{equation}
Consistently with assumption $(H')$,
we select, without loss of generality, the endpoints of the intervals $I_{i}$ in such a manner that
$a(x)\not\equiv0$ on each of the subintervals $\mathopen[\sigma_{i},\tau_{i}\mathclose]$ and, moreover,
$a(x)\not\equiv0$ on all left neighborhoods of $\sigma_{i}$ and on all right neighborhoods of $\tau_{i}$.

In order to explain the rule that we have decided to follow so as to determine the endpoints of the intervals,
let us consider the following weight function on the interval $I = \mathopen[0,7\pi\mathclose]$
\begin{equation*}
a(x) =
\begin{cases}
\, \sin x,     & \text{if }  x\in \mathopen[0,\pi\mathclose]\cup \mathopen[3\pi,4\pi\mathclose] \cup \mathopen[6\pi,7\pi\mathclose];\\
\, 0,          & \text{if }  x \in \mathopen[\pi,3\pi\mathclose] \cup \mathopen[4\pi,6\pi\mathclose].
\end{cases}
\end{equation*}
Among the various possibilities that one could adopt to choose the endpoints of the intervals according to condition $(H')$,
the following choice would fit with the above convention:
$\sigma_{1} = 0$, $\tau_{1} = 3\pi$, $\sigma_{2} = 4\pi$, $\tau_{2} = 7\pi$.

\noindent
To discuss another example, let us consider a
function with a graph as that of Figure~\ref{fig-02}.
It is clear that it satisfies \eqref{eq-a+} and \eqref{eq-a-},
provided that we adopt a suitable choice of the points $\sigma_{i}$ and $\tau_{i}$.
Typically, we shall proceed in the following manner: if there is an interval where $a(x)\equiv 0$ between
an interval where $a(x) > 0$ and an interval where $a(x) < 0$, we choose $\sigma_{i}$ and $\tau_{i}$ in such a way
that $a(x) < 0$ on $\mathopen]\tau_{i},\sigma_{i+1}\mathclose[$ and we merge the interval where $a(x)\equiv 0$
to an adjacent interval where $a(x)\geq 0$.

\medskip

We need also to introduce a further notation. For any weight function $a(x)$ satisfying $(H')$ (with the endpoints
$\sigma_{i}, \tau_{i}$ chosen as described above), we set
\begin{equation*}
a_{\mu}(x) := a^{+}(x) - \mu a^{-}(x),
\end{equation*}
where $\mu > 0$ is a parameter. Notice that $a(x) = a_{\mu}(x)$ for $\mu =1$ and, moreover, for every
$\mu > 0$, it holds that $a_{\mu}(x)$ satisfies $(H')$ with the same $\sigma_{i}$ and $\tau_{i}$ chosen for $a(x)$.

The introduction of the parameter $\mu$ is made only with the purpose to clarify the role of the negative humps
of $a(x)$ in order to produce multiplicity results. In other words, when we require that $\mu > 0$ is sufficiently large,
we have a more precise manner to express the intuitive fact that the negative humps of $a(x)$ are great enough.

Now we are in a position to present our main multiplicity result for the boundary value problem
\begin{equation}\label{two-pointBVPmu}
\begin{cases}
\, u''+a_{\mu}(x)g(u)=0 \\
\, u(0)=u(L)=0.
\end{cases}
\end{equation}
Recall that we are assuming that $a\colon I\to {\mathbb{R}}$
is a continuous weight function satisfying \eqref{eq-a+}, \eqref{eq-a-} and $(H')$ (with the above convention)
and $g(s)$ is continuous and satisfying \eqref{eq-g}.

\begin{theorem}\label{th-5.3}
Suppose that
\begin{equation*}
g_{0} < \lambda_{0}
\quad\mbox{and} \quad
a(x)\not\equiv 0 \mbox{ on each } I_{i} \mbox{ with }
g_{\infty} > \max_{i=1,\ldots,n} \lambda_{1}^{i}.
\end{equation*}
Then there exists $\mu^{*}>0$ such that, for each $\mu > \mu^{*}$,
problem \eqref{two-pointBVPmu} has at least $2^{n} -1$ positive solutions.
\end{theorem}

\begin{proof} From $g_{0}<\lambda_{0}$, we can choose $\delta>0$ such that $g_{0}  <\lambda_{0} -\delta$.
Let $r_{0}>0$ be as in Lemma~\ref{lemmar} and fix $0<r \leq r_{0}$ such that
\begin{equation}\label{eq-delta}
\dfrac{g(s)}{s}<\lambda_{0}-\delta, \quad \forall \, 0<s\leq r.
\end{equation}
Let $R^{*}>0$ as in Lemma~\ref{lemmaR} and fix $R\geq R^{*}$.

Let $\mathcal{I}\subseteq\{1,\ldots,n\}$. Using the same notation as in Section~\ref{section4},
consider the open and unbounded set $\Gamma^{\mathcal{I}}$.
Moreover, consider an arbitrary Carath\'{e}odory function $h\colon I\times{\mathbb{R}}^{+}\to{\mathbb{R}}$ such that
\begin{equation}\label{cond_h}
\begin{aligned}
   &  h(x,s) \geq a(x)g(s), \quad \text{a.e. } x\in \bigcup_{i\in {\mathcal{I}}} I_{i}, \; \forall \, s\geq 0;
\\ &  h(x,s) = a(x)g(s),    \quad \text{a.e. } x\in I \setminus \bigcup_{i\in {\mathcal{I}}} I_{i}, \; \forall \, s\geq 0;
\end{aligned}
\end{equation}
and, as usual, define the completely continuous operator $\Psi_{h}\colon\mathcal{C}(I)\to\mathcal{C}(I)$,
\begin{equation*}
(\Psi_{h} u)(x):=\int_{I}G(x,\xi)\tilde{h}(\xi,u(\xi))~\!d\xi,
\end{equation*}
where
\begin{equation*}
\tilde{h}(x,s):=
\begin{cases}
\, h(x,s), & \text{if } s\geq0; \\
\, h(x,0), & \text{if } s\leq0.
\end{cases}
\end{equation*}
We know that every fixed point of $\Psi_{h}$ is a non-negative solution of
\begin{equation}\label{two-pointBVPh}
\begin{cases}
\, u''+h(x,u)=0 \\
\, u(0)=u(L)=0.
\end{cases}
\end{equation}

To prove the claim, we use Theorem~\ref{Th-Multiplicity}. In particular we have to show that
the triplet $(I-\Psi_{h},\Gamma^{\mathcal{I}},0)$ is admissible
for each Carath\'{e}odory function $h$ satisfying \eqref{cond_h}
and for each $\emptyset\neq\mathcal{I}\subseteq\{1,\ldots,n\}$.

By the choice of $R \geq R^{*}$ and by the convexity of the solution of \eqref{two-pointBVPh} on
each interval contained in $I \setminus \bigcup_{i=1}^{n} I_{i}$,
we know that every fixed point $u$ of $\Psi_{h}$ is contained in the open ball $B(0,R)$,
with $R > 0$ independent of the particular choice of $h(x,s)$ (see Lemma~\ref{lemmaRJ}).
Consequently it is sufficient to prove that $\Psi_{h}$ has no fixed points in $\partial(\Gamma^{\mathcal{I}}\cap B(0,R))$,
for $\mu$ sufficiently large.

First we note that if $\mathcal{I}=\{1,\ldots,n\}$, there is nothing to prove,
since all fixed points in $\Gamma^{\mathcal{I}}=\mathcal{C}(I)$ are contained in $B(0,R)$.
Then fix $\emptyset\neq\mathcal{I}\subsetneq\{1,\ldots,n\}$.
By contradiction, suppose that there is a fixed point $u$ of $\Psi_{h}$ in \mbox{$\partial(\Gamma^{\mathcal{I}}\cap B(0,R))$}.
Due to what we have just remarked, this is equivalent to assuming
the existence of a solution $u$ of \eqref{two-pointBVPh} with
\begin{equation*}
\max_{x\in I_{k}} u(x) = r, \quad \text{ for some index } \, k\in\{1,\ldots,n\}\setminus\mathcal{I},
\end{equation*}
and such that $\max_{x\in I} u(x) < R$. Clearly $u\not\equiv0$
and, moreover, by the concavity of $u(x)$ in $I_{k}$, we also have
\begin{equation}\label{eq-conc}
u(x) \geq \dfrac{r}{\tau_{k}-\sigma_{k}} \min\{x-\sigma_{k},\tau_{k}-x\}, \quad \forall \, x\in I_{k}.
\end{equation}

In order to prove that our assumption is contradictory and hence that the topological degree is well defined,
we split our argument into three steps.

\medskip

\noindent
\textit{Step 1: A priori bounds for $|u'(x)|$ on $I_{k}$.} This part of the proof follows by adapting a similar estimate obtained in
Theorem~\ref{Th+LinCond}. Notice that $h(x,u(x)) = a(x) g(u(x)) = a^{+}(x) g(u(x))$, for a.e.~$x\in I_{k}$. Hence
\begin{equation*}
|u''(x)| \leq \gamma_r(x):= a^{+}(x)\max_{0\leq s\leq r}g(s), \quad \text{a.e. } x\in I_{k},
\end{equation*}
and, therefore
\begin{equation*}
|u'(y_{1}) - u'(y_{2})| \leq \|\gamma_r\|_{L^{1}(I_{k})}, \quad \forall \, y_{1}, y_{2} \in I_{k}.
\end{equation*}
Let $\hat{x} \in I_{k} = \mathopen{[}\sigma_{k},\tau_{k}\mathclose{]}$ be such that
$|u'(\hat{x})| \leq r/(\tau_{k} - \sigma_{k})$ (otherwise we have an easy contradiction like in the proof of Theorem~\ref{Th+LinCond}).
Hence for all $x\in I_{k}$
\begin{equation}\label{eq-MK}
|u'(x)| \leq |u'(\hat{x})| + |u'(x) - u'(\hat{x})| \leq \dfrac{r}{\tau_{k} - \sigma_{k}} + \|\gamma_r\|_{L^{1}(I_{k})} =: M_{k}.
\end{equation}

\medskip

\noindent
\textit{Step 2: Lower bounds for $u(x)$ on the boundary of $I_{k}$.}
Let $\varphi_{k}$, with $\|\varphi_{k}\|_{\infty}=1$, be the positive eigenfunction on $I_{k}=\mathopen[\sigma_{k},\tau_{k}\mathclose]$ of
\begin{equation*}
\begin{cases}
\, \varphi''+\lambda_{1}^{k}a^{+}(x)\varphi=0 \\
\, \varphi|_{\partial_{I_{k}}}=0,
\end{cases}
\end{equation*}
where $\lambda_{1}^{k}$ is the first eigenvalue.
Then $\varphi_{k}(x)\geq0$, for all $x\in I_{k}$, $\varphi_{k}(x)>0$, for all $x\in\mathopen]\sigma_{k},\tau_{k}\mathclose[$,
and $\varphi_{k}'(\sigma_{k})>0>\varphi_{k}'(\tau_{k})$ (hence $\|\varphi_{k}'\|_{\infty}>0$).

By \eqref{eq-delta} and $\lambda_{0} \leq \lambda_{1}^{k}$, we know that
\begin{equation*}
g(s)<(\lambda_{1}^{k}-\delta) s, \quad \forall \, 0<s\leq r.
\end{equation*}
Then, by \eqref{eq-conc}, we have
\begin{equation*}
\begin{aligned}
   & \|\varphi_{k}'\|_{\infty}(u(\sigma_{k})+u(\tau_{k})) \geq
\\ & \geq u(\sigma_{k})|\varphi_{k}'(\sigma_{k})|+u(\tau_{k})|\varphi_{k}'(\tau_{k})|
= u(\sigma_{k})\varphi_{k}'(\sigma_{k})-u(\tau_{k})\varphi_{k}'(\tau_{k})
\\ & = \Bigl{[}u'(x)\varphi_{k}(x)-u(x)\varphi_{k}'(x)\Bigr{]}_{x=\sigma_{k}}^{x=\tau_{k}}
\\ & = \int_{\sigma_{k}}^{\tau_{k}}\dfrac{d}{dx}\Bigl{[}u'(x)\varphi_{k}(x)-u(x)\varphi_{k}'(x)\Bigr{]} ~\!dx
\\ & = \int_{I_{k}}\Bigl{[}u''(x)\varphi_{k}(x)-u(x)\varphi_{k}''(x)\Bigr{]} ~\!dx
\\ & = \int_{I_{k}}\Bigl[-h(x,u(x))\varphi_{k}(x)+u(x)\lambda_{1}^{k}a^{+}(x)\varphi_{k}(x)\Bigr] ~\!dx
\\ & = \int_{I_{k}}\Bigl[\lambda_{1}^{k}u(x)-g(u(x))\Bigr]a^{+}(x)\varphi_{k}(x) ~\!dx
\\ & > \int_{I_{k}}\delta \, \biggl(\dfrac{r}{\tau_{k}-\sigma_{k}} \min\{x-\sigma_{k},\tau_{k}-x\}\biggr)a^{+}(x)\varphi_{k}(x) ~\!dx
\\ & = r\, \biggl[\dfrac{\delta}{\tau_{k}-\sigma_{k}}\int_{I_{k}}\min\{x-\sigma_{k},\tau_{k}-x\}a^{+}(x)\varphi_{k}(x) ~\!dx\biggr].
\end{aligned}
\end{equation*}

Hence, from the above inequality, we conclude that there exists a constant $c_{k}>0$, depending on $\delta$,
$I_{k}$ and $a^{+}(x)$, but independent of $u(x)$ and $r$, such that
\begin{equation*}
u(\sigma_{k}) + u(\tau_{k}) \geq {c}_{k}r > 0.
\end{equation*}
As a consequence of the above inequality, we have that at least one of the two inequalities
\begin{equation}\label{eq-ineq}
0 < {c}_{k}r/2 \leq u(\tau_{k}) \leq r, \qquad  0 < {c}_{k}r/2 \leq u(\sigma_{k}) \leq r,
\end{equation}
holds.

\medskip

\noindent
\textit{Step 3: Contradiction on an adjacent interval for $\mu$ large.}
Just to fix a case for the rest of the proof, suppose that the first inequality in \eqref{eq-ineq} is true.
In such a situation, we necessarily have $\tau_{k} < L$ (as $u(L) = 0$).
Now we focus our attention on the right-adjacent interval $\mathopen{[}\tau_{k},\sigma_{k+1}\mathclose{]}$, where $a(x)\leq 0$.
Recall also that, by the convention
we have adopted in defining the intervals $I_{i}$, we have that $a(x)$ is not identically zero on all right neighborhoods of $\tau_{k}$.

Since $g(s) > 0$ for all $s > 0$, we can introduce the positive constant
\begin{equation*}
\nu_{k}:=\min_{\frac{{c}_{k}r}{4}\leq s\leq R}g(s)>0
\end{equation*}
and define
\begin{equation*}
\delta^{+}_{k}:= \min\biggl\{\sigma_{k+1}-\tau_{k},\dfrac{{c}_{k}r}{4M_{k}}\biggr\}>0,
\end{equation*}
where $M_{k}>0$ is the bound for $|u'|$ obtained in \eqref{eq-MK} of \textit{Step 1}.
Then, by the convexity of $u(x)$ on $\mathopen{[}\tau_{k},\sigma_{k+1}\mathclose{]}$, we have that $u(x)$
is bounded from below by the tangent line at $(\tau_{k},u(\tau_{k}))$, with slope $u'(\tau_{k}) \geq -M_{k}$. Therefore,
\begin{equation*}
\dfrac{{c}_{k}r}{4}\leq u(x)\leq R, \quad \forall \, x\in\mathopen[\tau_{k},\tau_{k}+\delta^{+}_{k} \mathclose].
\end{equation*}

We prove that for $\mu>0$ sufficiently large $\max_{x\in\mathopen[\tau_{k},\sigma_{k+1}\mathclose]}u(x)>R$
(which is a contradiction to the upper bound for $u(x)$).

Consider the interval $\mathopen[\tau_{k},\tau_{k}+\delta^{+}_{k}\mathclose]\subseteq\mathopen[\tau_{k},\sigma_{k+1}\mathclose]$.
For all $x\in \mathopen[\tau_{k},\tau_{k}+\delta^{+}_{k}\mathclose]$ we have
\begin{equation*}
u'(x) = u'(\tau_{k})+\int_{\tau_{k}}^{x}\mu a^{-}(\xi)g(u(\xi))~\!d\xi
      \geq -M_{k}+\mu\nu_{k}\int_{\tau_{k}}^{x} a^{-}(\xi)~\!d\xi,
\end{equation*}
then
\begin{equation*}
u(x) =     u(\tau_{k})+\int_{\tau_{k}}^{x}u'(\xi)~\!d\xi
     \geq \dfrac{{c}_{k}r}{2}-M_{k}(x-\tau_{k})+\mu\nu_{k}\int_{\tau_{k}}^{x}\biggr(\int_{\tau_{k}}^{s} a^{-}(\xi)~\!d\xi\biggr) ~\!ds.
\end{equation*}
Hence, for $x=\tau_{k}+\delta^{+}_{k} $,
\begin{equation*}
R \geq  u(\tau_{k}+\delta^{+}_{k} ) \geq
\dfrac{{c}_{k}r}{2}-M_{k}\delta^{+}_{k}
+ \mu\nu_{k}\int_{\tau_{k}}^{\tau_{k}+\delta^{+}_{k} }\biggl(\int_{\tau_{k}}^{s} a^{-}(\xi)~\!d\xi \biggr) ~\!ds.
\end{equation*}
This gives a contradiction if $\mu$ is sufficiently large, say
\begin{equation*}
\mu > \mu_{k}^{+} := \dfrac{R + M_{k} L}{\nu_{k}{A}^{+}_{k}},
\end{equation*}
where we have set
\begin{equation*}
{A}^{+}_{k}:= \int_{\tau_{k}}^{\tau_{k}+\delta^{+}_{k} }\biggl(\int_{\tau_{k}}^{s} a^{-}(\xi)~\!d\xi \biggr) ~\!ds>0,
\end{equation*}
recalling that $\int_{\tau_{k}}^x  a^{-}(\xi) ~\!d\xi > 0$ for each $x\in \mathopen]\tau_{k},\sigma_{k+1}\mathclose]$.

A similar argument (with obvious modifications) applies if the second inequality in \eqref{eq-ineq} is true
(in such a case, we must have $\sigma_{k} > 0$, as $u(0) = 0$). This time we focus our attention
on the left-adjacent interval $\mathopen{[}\tau_{k-1},\sigma_{k}\mathclose{]}$ where $a(x)\leq 0$. Recall also that, by the convention
we have adopted in defining the intervals $I_{i}$, we have that $a(x)$ is not identically zero on all left neighborhoods of $\sigma_{k}$.

\noindent
If we define
\begin{equation*}
\delta^{-}_{k}  := \min\biggl\{\sigma_{k}-\tau_{k-1},\dfrac{{c}_{k}r}{4M_{k}}\biggr\}>0,
\end{equation*}
we obtain the same contradiction for
\begin{equation*}
\mu > \mu_{k}^{-} := \dfrac{R + M_{k} L}{\nu_{k}{A}^{-}_{k}},
\end{equation*}
where we have set
\begin{equation*}
{A}^{-}_{k}:= \int_{\sigma_{k}-\delta^{-}_{k} }^{\sigma_{k}}\biggl(\int_{s}^{\sigma_{k}} a^{-}(\xi) ~\!d\xi \biggr) ~\!ds.
\end{equation*}

At the end, we define
\begin{equation*}
\mu^{*}:= \max_{k=1,\ldots,n} \mu^{\pm}_{k}
\end{equation*}
and we apply Theorem~\ref{Th-Multiplicity} with $\mu > \mu^{*}$.
The proof is complete.
\end{proof}

An immediate consequence of Theorem~\ref{th-5.3} is the following result which generalizes \cite[Theorem~2.1]{GaudenziHabetsZanolin4}.

\begin{corollary}\label{cor-5.3}
Let $g\colon {\mathbb{R}}^{+} \to {\mathbb{R}}^{+}$ be a continuous function such that
$g(0) = 0$ and $g(s) > 0$ for all $s > 0$.
Suppose also that
\begin{equation*}
\lim_{s\to 0^{+}}\dfrac{g(s)}{s} = 0 \quad \text{ and } \quad \lim_{s\to +\infty}\dfrac{g(s)}{s} = +\infty.
\end{equation*}
Let $a^{\pm}\colon\mathopen{[}0,L\mathclose{]}\to{\mathbb{R}}^{+}$ be continuous functions such that for some
$0 = \sigma_{1} < \tau_{1} < \sigma_{2} < \tau_{2} < \ldots <\sigma_{n} < \tau_{n} = L$ it holds
\begin{equation*}
\begin{aligned}
& a^{+}(x)\not\equiv 0, \; a^{-}(x) \equiv 0, \text{ on } \; \mathopen{[}\sigma_{i},\tau_{i}\mathclose{]}, \; i=1,\ldots, n;\\
& a^{-}(x)\not\equiv 0, \; a^{+}(x) \equiv 0, \text{ on } \; \mathopen{[}\tau_{i},\sigma_{i+1}\mathclose{]}, \; i=1,\ldots, n-1.
\end{aligned}
\end{equation*}
Then there exists $\mu^{*}>0$ such that, for each $\mu > \mu^{*}$,
problem \eqref{two-pointBVPmu} has at least $2^{n} -1$ positive solutions.
\end{corollary}

Clearly, under the assumptions of Corollary~\ref{cor-5.3}, condition $(H')$ holds with $0 = \tau_{0} = \sigma_{1}$
and $\tau_{n} = \sigma_{n+1} = L$. Variants of the result can be easily stated if $0 < \sigma_{1}$ or $\tau_{n} < L$.
In any case, the number of positive solutions is at least $2^{n}-1$, where $n$ is the number of positive humps.

\medskip

A simple consequence of Theorem~\ref{th-5.3}  can be obtained when $g_{0} = 0$ and $g_{\infty} > 0$
(not necessarily $g_{\infty} = + \infty$ as in Corollary~\ref{cor-5.3})
for the problem
\begin{equation*}
\begin{cases}
\, u''+ (\nu a^{+}(x) - \mu a^{-}(x))g(u)=0 \\
\, u(0)=u(L)=0.
\end{cases}
\end{equation*}
Indeed, in such a case, we have that there exists $\nu^{*} > 0$ such that for each
$\nu > \nu^{*}$ there exists $\mu^{*} = \mu^{*}(\nu) > 0$ so that
there are at least $2^{n}-1$ positive solutions for each $\mu > \mu^{*}$.

\subsection{Radially symmetric solutions}\label{subsec-5.3}

Let $\|\cdot\|$ be the Euclidean norm in ${\mathbb{R}}^{N}$ (for $N \geq 2$) and let
\begin{equation*}
\Omega:= B(0,R_{2})\setminus B[0,R_{1}] = \{x\in {\mathbb{R}}^{N} \colon R_{1} < \|x\| < R_{2}\}
\end{equation*}
be an open annular domain, with $0 < R_{1} < R_{2}$.
Let ${\mathcal{A}}_{\mu}\colon \mathopen{[}R_{1},R_{2}\mathclose{]}\to {\mathbb{R}}$ be a continuous function defined as
\begin{equation*}
{\mathcal{A}}_{\mu}(r):= {\mathcal{A}}^{+}(r) - \mu {\mathcal{A}}^{-}(r), \quad \mu > 0.
\end{equation*}

We consider the problem of existence of positive solutions for the Dirichlet boundary value problem
\begin{equation}\label{eq-dir}
\begin{cases}
\, -\Delta \,{\mathcal{U}} = {\mathcal{A}}_{\mu}(\|x\|)\,g({\mathcal{U}}) & \text{ in } \Omega\\
\, {\mathcal{U}} = 0                                                      & \text{ on } \partial\Omega,
\end{cases}
\end{equation}
namely classical solutions such that ${\mathcal{U}}(x) > 0$ for all $x\in \Omega$.
If we look for radially symmetric solutions of \eqref{eq-dir}, we are led to the study of the two-point boundary value problem
\begin{equation}\label{eq-rad}
v''(r) + \dfrac{N-1}{r}v'(r) + {\mathcal{A}}_{\mu}(r) g(v(r)) = 0, \quad v(R_{1}) = v(R_{2}) = 0.
\end{equation}
Indeed, if $v(r)$ is any solution of \eqref{eq-rad}, then ${\mathcal{U}}(x):= v(\|x\|)$ is a solution of \eqref{eq-dir}.
Using the standard change of variable
\begin{equation*}
t = h(r):= \int_{R_{1}}^r \xi^{1-N} ~\!d\xi,
\end{equation*}
it is possible to
transform \eqref{eq-rad} into the equivalent problem
\begin{equation}\label{eq-rad1}
u''(t) +  r(t)^{2(N-1)}{\mathcal{A}}_{\mu}(r(t)) g(u(t)) = 0, \quad u(0) = u(L) = 0,
\end{equation}
for
\begin{equation*}
u(t)=v(r(t)),
\end{equation*}
with the positions
\begin{equation*}
L:= \int_{R_{1}}^{R_{2}} \xi^{1-N} ~\!d\xi \quad \mbox{ and } \quad r(t):= h^{-1}(t).
\end{equation*}
Clearly, problem \eqref{eq-rad1} is of the same form of \eqref{two-pointBVPmu} with
\begin{equation*}
a_{\mu}(t):= r(t)^{2(N-1)} {\mathcal{A}}_{\mu}(r(t)).
\end{equation*}
Then, the following results hold, for every continuous function
$g \colon {\mathbb{R}}^{+} \to {\mathbb{R}}^{+}$ such that
$g(0) = 0$ and $g(s) > 0$ for all $s > 0$ and satisfying
\begin{equation*}
\lim_{s\to 0^{+}}\dfrac{g(s)}{s} = 0 \quad \mbox{ and } \quad \lim_{s\to +\infty}\dfrac{g(s)}{s} = +\infty.
\end{equation*}

\begin{theorem}\label{th-radial1}
Suppose that ${\mathcal{A}}_{\mu}(r)$ changes its sign in $\mathopen{[}R_{1},R_{2}\mathclose{]}$ at most a finite number of times
and ${\mathcal{A}}^{+}(r)\not\equiv 0$.
Then, for every $\mu \geq 0$,
problem \eqref{eq-dir} has at least a positive solution.
\end{theorem}

\begin{theorem}\label{th-radial2}
Suppose that for some
$R_{1} = \sigma_{1} < \tau_{1} < \sigma_{2} < \tau_{2} < \ldots <\sigma_{n} < \tau_{n} = R_{2}$ it holds:
\begin{equation*}
\begin{aligned}
& {\mathcal{A}}^{+}(x)\not\equiv 0, \; {\mathcal{A}}^{-}(x) \equiv 0,   \text{ on } \; \mathopen{[}\sigma_{i},\tau_{i}\mathclose{]}, \; i=1,\ldots, n;\\
& {\mathcal{A}}^{-}(x)\not\equiv 0, \; {\mathcal{A}}^{+}(x) \equiv 0,   \text{ on } \; \mathopen{[}\tau_{i},\sigma_{i+1}\mathclose{]}, \; i=1,\ldots, n-1.\
\end{aligned}
\end{equation*}
Then there exists $\mu^{*}>0$ such that, for each $\mu > \mu^{*}$,
problem \eqref{eq-dir} has at least $2^{n} -1$ positive solutions.
\end{theorem}

Theorem~\ref{th-radial1} and Theorem~\ref{th-radial2} can be seen as an extension of the classical existence result of Bandle, Coffman and Marcus
\cite{BandleCoffmanMarcus}
to the case of a general sign-changing weight. It could be interesting to investigate under which supplementary assumptions
the above results are sharp (that is, providing exactly one positive solution or exactly $2^{n}-1$ positive solutions, respectively).

As a comment about the sign conditions on ${\mathcal{A}}_{\mu}(r)$, we observe that our results
apply to weight functions which may vanish in some sub-intervals of $\mathopen{[}R_{1},R_{2}\mathclose{]}$ (even in infinitely many sub-intervals),
see Remark~\ref{rem-4.1}.
Concerning the continuous nonlinearity $g(s)$, we notice that, besides the positivity and the conditions for $s\to 0^{+}$ and for $s\to +\infty$,
no other assumptions (like smoothness, monotonicity or homogeneity) are required.

\subsection{Final remarks}\label{subsec-5.4}

For the study of problem \eqref{two-pointBVPag} we have confined ourselves to the case of a continuous weight function $a(x)$.
Since the general results for problem \eqref{two-pointBVP} have been obtained under general Carath\'{e}odory assumptions on $f(x,s)$,
we can deal with the case of $a \in L^{1}(I)$, too.
With this respect, Theorem~\ref{th-5.1} and Theorem~\ref{th-5.2} are still valid provided that the assumption $a(x)\not\equiv0$ on $I_{i}$
is meant in the sense that $a(x)\geq 0$ for a.e.~$x\in I_{i}$ and $\int_{I_{i}} a(x) ~\!dx > 0$. Concerning the variant of Theorem~\ref{th-5.3}
for $a_{\mu}(x) = a^{+}(x) - \mu a^{-}(x)$, with $a^{\pm} \in L^{1}(I)$ and $a^{\pm}\geq 0$ almost everywhere, we claim that our result still holds
provided that the endpoints of the intervals are selected so that $\int_{\tau_{k}}^x a^{-}(\xi) ~\!d\xi > 0$ for all
$x\in \mathopen]\tau_{k},\sigma_{k+1}\mathclose]$ and $\int_x^{\sigma_{k}} a^{-}(\xi) ~\!d\xi > 0$ for all
$x\in \mathopen[\tau_{k-1},\sigma_{k}\mathclose[$ (for each $k=1,\ldots,n$). In this manner, the constants ${A}^{\pm}_{k}$
in \textit{Step 3} of the proof of Theorem~\ref{th-5.3} are all strictly positive. All the other parts of the proof are exactly the same.

In \cite{GaudenziHabetsZanolin3} a class of measurable weight functions which are possibly singular at the endpoints of the interval $I$
is considered. More precisely, therein one can consider a function $a\in L^{1}_{\text{loc}}(I)$ such that $\int_{I}x(L-x)|a(x)| ~\!dx < +\infty$.
The possibility of dealing with weight functions which are not in $L^{1}(I)$ depends by the method of proof in \cite{GaudenziHabetsZanolin3}
based on the search of fixed points for the operator associated with the Green function. Since in this paper we follow exactly the
same approach, we can also deal with such a wider class of weight functions.

The approach that we have followed in the present work can be adapted to the study of different boundary value problems.
For instance, like in \cite{LanWebb}, one can consider mixed boundary conditions like
$u'(0) = u(L) = 0$ or $u(0) = u'(L) = 0$.

\appendix
\section{Appendix}\label{section6}

In this appendix we present a general version of the Leray-Schauder degree. For more details we refer to \cite{Mawhin,Nussbaum2,Nussbaum3}
and the references therein.

Let $X$ be a normed linear space, $\Omega\subseteq X$ an open subset and $z\in X$.
Consider a continuous map $\phi\colon \Omega\to X$ such that $S_{z}:=\left\{x\in\Omega\colon x-\phi(x)=z\right\}$ is a compact set (possibly empty)
and such that there exists an open neighborhood $V$ of $S_{z}$ with $\overline{V}\subseteq\Omega$ such that $\phi|_{\overline{V}}$ is compact.
If all the previous assumptions are satisfied, the triplet $(Id-\phi,\Omega,z)$ is called \textit{admissible}.

\noindent
To the admissible triplet $(Id-\phi,\Omega,z)$ we associate the integer
\begin{equation*}
deg(Id-\phi,\Omega,z),
\end{equation*}
called the \textit{Leray-Schauder degree of $Id-\phi$ on $\Omega$ in $z$}, satisfying the following three axioms.
\begin{enumerate}
    \item [(LS1)] \textit{Additivity}.
            If $\Omega_{1},\Omega_{2}\subseteq \Omega$ are open and disjoint subsets and $S_{z}\subseteq \Omega_{1}\cup\Omega_{2}$, then
            \begin{equation*}
            deg(Id-\phi,\Omega,z)=deg(Id-\phi,\Omega_{1},z)+deg(Id-\phi,\Omega_{2},z).
            \end{equation*}
    \item [(LS2)] \textit{Homotopy invariance}.
            Let $U\subseteq X\times\mathopen{[}a,b\mathclose{]}$ be an open subset
            (typically $U=\Omega\times\mathopen{[}a,b\mathclose{]}$, with $\Omega\subseteq X$).
            Let $h\colon U\to X$ be a continuous map.
            Define $h_{\lambda}(x):=h(x,\lambda)$ and $U_{\lambda}:=\left\{x\in X\colon (x,\lambda)\in U\right\}$.
            Suppose that the set $\Sigma:=\left\{(x,\lambda)\in U\colon x-h_{\lambda}(x)=z\right\}$
            is compact (possibly empty). Assume that there exists an open neighborhood $W$ of $\Sigma$
            such that $h|_{\overline{W}}$ is a compact map.
            Then $deg(Id-h_{\lambda},U_{\lambda},z)$ is constant with respect to $\lambda\in\mathopen{[}a,b\mathclose{]}$.
    \item [(LS3)] \textit{Normalization}.
            \begin{equation*}
            deg(Id,\Omega,z):=\begin{cases}
            \, 1, & \text{if } z\in \Omega;\\
            \, 0, & \text{if } z\in X\setminus\overline{\Omega}.
            \end{cases}
            \end{equation*}
\end{enumerate}
Using the axioms, one can easily prove that if $deg(Id-\phi,\Omega,z)\neq0$,
then there exists $\hat{x}\in\Omega$ such that $\hat{x}-\phi(\hat{x})=z$.

In the framework of this paper we need a simpler version of the general topological degree described above.
Namely, in our applications we consider a completely continuous operator $\phi\colon X\to X$ and an open set $\Omega\subseteq X$.
Since we focus on the existence of fixed points of a map $\phi$, we take $z=0$ and we are interested in studying the integer
\begin{equation*}
deg(Id-\phi,\Omega,0).
\end{equation*}
To prove the admissibility of $(Id-\phi,\Omega,0)$ it is sufficient either to establish that
\begin{equation*}
S_{0} = \left\{x\in\Omega\colon x-\phi(x)=0\right\}
\end{equation*}
is compact
or, equivalently, to show that the set of all possible fixed points of $\phi$ in the whole space $X$ is contained in an open and bounded set $W$
satisfying $x-\phi(x)\neq 0$, for all $x\in\partial(\Omega\cap W)$.

\medskip

Now we state two theorems which are useful for our applications.

\begin{theorem}\label{deFigueiredo}
Let $X$ be a normed linear space and $\Omega\subseteq X$ an open and bounded subset. Let $\phi\colon\overline{\Omega}\to X$ be a compact map.
If $F\colon\overline{\Omega}\times\mathopen[0,+\infty\mathclose[\to X$ is a compact map such that
\begin{itemize}
\item [(i)] $F(x,0)=\phi(x)$, for all $x\in\partial\Omega$;
\item [(ii)] $F(x,\alpha)\neq x$, for all $x\in\partial\Omega$ and $\alpha\geq0$;
\item [(iii)] there exists $\alpha_{0}\geq0$ such that $F(x,\alpha)\neq x$, for all $x\in\overline{\Omega}$ and $\alpha\geq \alpha_{0}$;
\end{itemize}
then
\begin{equation*}
deg(Id-\phi,\Omega,0)=0.
\end{equation*}
Moreover, if there exists $v\in X\setminus\{0\}$ such that $x\neq\phi(x)+\alpha v$,
for all $x\in\partial\Omega$ and $\alpha\geq0$, then conditions $(i)$, $(ii)$, $(iii)$ are satisfied.
\end{theorem}

We omitted the easy proof. In \cite[pp.~67-68]{deFigueiredo} the author proves the statement for an open ball.

Now we consider open and possibly unbounded sets, as in the context of our applications.

\begin{theorem}\label{deFigueiredo2}
Let $X$ be a normed linear space and $\Omega\subseteq X$ an open subset. Let $\phi\colon X\to X$ be a continuous map
and $F\colon X\times\mathopen[0,+\infty\mathclose[\to X$ a completely continuous map.
Suppose that
\begin{itemize}
 \item [(i)] $F(x,0)=\phi(x)$, for all $x\in X$;
 \item [(ii)] for all $\alpha\geq0$ there exists $R_{\alpha}>0$ such that
         if there exist $x\in\overline{\Omega}$ and $\zeta\in[0,\alpha]$ such that $x=F(x,\zeta)$, then $\|x\|\leq R_{\alpha}$ and $x\in\Omega$;
 \item [(iii)] there exists $\alpha_{0}\geq0$ such that $x\neq F(x,\alpha)$, for all $x\in\overline{\Omega}$ and $\alpha\geq\alpha_{0}$.
\end{itemize}
Then the triplet $(Id-\phi,\Omega,0)$ is admissible and $deg(Id-\phi,\Omega,0)=0$.
\end{theorem}

\begin{proof}
Without loss of generality, we can assume that $R_{\alpha'}<R_{\alpha''}$, if $\alpha'<\alpha''$.
The set $A:=B(0,R_{\alpha_{0}+1})\cap\Omega$ is open and bounded,
and, by conditions $(ii)$ and $(iii)$, it contains all possible fixed points of $F(\cdot,\alpha)$ in $\overline{\Omega}$.
Using also $(i)$, we have that
\begin{equation*}
deg(Id-\phi,\Omega,0)=deg(Id-F(\cdot,0),\Omega,0)=deg(Id-F(\cdot,0),A,0).
\end{equation*}
Taking $h_{\alpha}:=F(\cdot,\alpha)$, $\alpha\in\mathopen[0,\alpha_{0}\mathclose]$, as admissible homotopy, by $(ii)$ and (LS2) we obtain that
\begin{equation*}
deg(Id-F(\cdot,\alpha),A,0)=const., \qquad 0\leq\alpha\leq\alpha_{0}.
\end{equation*}
By $(iii)$, we conclude that
\begin{equation*}
deg(Id-\phi,\Omega,0)=deg(Id-F(\cdot,\alpha_{0}),A,0)=0.
\end{equation*}
\end{proof}

\bibliographystyle{elsart-num-sort}
\bibliography{FeltrinZanolin_biblio}

\bigskip
\begin{flushleft}

{\small{\it Preprint}}

\end{flushleft}

\end{document}